\documentclass[11pt,  reqno]{amsart}
\usepackage{fullpage}

\usepackage{amsmath}
\usepackage{amsthm, mathrsfs, mathtools}
\usepackage{amsfonts}
\usepackage{amssymb}
\usepackage{amsaddr}
\usepackage{tikz-cd}
\usepackage[utf8x]{inputenc}
\usepackage{enumerate}
\usepackage[english]{babel}
\usepackage{esint}
\usepackage{hyperref}
\usepackage{bm}
\usepackage{bbm}
\usepackage[mathscr]{eucal}

\theoremstyle{plain}
\newtheorem{theorem}{Theorem}[section]
\newtheorem{lemma}[theorem]{Lemma}

\theoremstyle{definition}

\theoremstyle{remark}

\newtheorem{remark}[theorem]{Remark}

\numberwithin{equation}{section}

\newcommand{\ve}{\varepsilon}

\def\div{\,\mathrm{div}}

\newcommand{\dx}{\textup{d}x}
\newcommand{\dt}{\textup{d}t}
\newcommand{\ds}{\textup{d}s}
\newcommand{\dr}{\textup{d}r}

\definecolor{falured}{rgb}{0.5, 0.09, 0.09}
\definecolor{pinegreen}{rgb}{0.0, 0.47, 0.44}
\definecolor{denim}{rgb}{0.08, 0.38, 0.74}

\def\softd{{\leavevmode\setbox1=\hbox{d}%
		\hbox to 1.05\wd1{d\kern-0.4ex{\char039}\hss}}}

\allowdisplaybreaks[1]
\makeatletter
\def\@settitle{\begin{center}%
  \baselineskip14\p@\relax
    \huge
  \@title
  \end{center}%
}
\makeatother

\hypersetup{colorlinks=true, linkcolor=blue, citecolor=red, urlcolor=blue}

\begin{document}
	\title{Global weak solutions to the Cahn-Hilliard equation \\ 
    with degenerate mobility and singular diffusion}

\author[Monica Conti, Andrea Giorgini and Greta Ricchi]{Monica Conti, Andrea Giorgini and Greta Ricchi}

\address{Politecnico di Milano\\
Dipartimento di Matematica\\
Via E. Bonardi 9, I-20133 Milano, Italy\\
\href{mailto:monicaconti@polimi.it}{monica.conti@polimi.it},
\href{mailto:andrea.giorgini@polimi.it}{andrea.giorgini@polimi.it},
\href{mailto:greta.ricchi@polimi.it}{greta.ricchi@polimi.it}}

\begin{abstract}
We study the initial-boundary value problem for the Cahn-Hilliard equation with degenerate mobility and singular diffusion at pure phases. 
This model describes the dynamics of phase separation in polymer blends with associated Flory-Huggins-de Gennes free energy. 
We prove the existence of global weak solutions in three-dimensional bounded and smooth domains, assuming that the initial datum has finite energy. 
We show the classical regularity in $L^2(0,T;H^2(\Omega))$ for both the solution $u$ and the function $\phi(u)=\arcsin(u)$ through entropy estimates.
In addition, when $\Omega$ is convex, a new elliptic-type argument yields the refined regularity $u, \phi(u) \in L^4(0,T; H^2(\Omega))$.
\end{abstract}

\maketitle 

\setcounter{tocdepth}{1}
\tableofcontents

\section{Introduction and main results}

We consider the Cahn-Hilliard system
 \begin{align}
\partial_t u  &= \div(b(u) \nabla \mu ), \label{A1}\\
\mu&= - \div \left( a(u) \nabla u \right) + \dfrac{a'(u)}{2} |\nabla u |^2 + f(u) -\theta_0 u.\label{A2}
 \end{align} 
The problem \eqref{A1}--\eqref{A2} is posed in $\Omega_T:=\Omega \times (0,T)$, where $\Omega \subset \mathbbm{R}^3$ denotes a bounded domain and $T>0$ is arbitrarily chosen. 
The state variable is $u= u(x,t): \Omega_T \to \mathbbm{R}$, representing the difference of the concentrations of two components. In particular, the solution is subject to the constraint $-1\leq u(x,t) \leq 1$, where the values $\pm 1$ correspond to the pure phases. 

The function $b$ is the Onsager mobility and describes the intensity of diffusion. In this work, we consider a degenerate mobility, which vanishes at the pure phases $u=\pm1$.
More precisely, we assume the physically relevant form
\begin{equation}
\label{b}
 b(u)= 1-u^2.
\end{equation}
The function $a$ is a concentration-dependent coefficient, accounting for nonlinear diffusive effects. In particular, we focus on the ``singular" case
    \begin{equation}
        \label{a}
        a(u)= \dfrac{2}{1-u^2}.
    \end{equation} 
We also introduce the Helmholtz free energy density of the system
    \begin{equation}
        \label{Psi}
        \Psi(u)= F(u)-\dfrac{\theta_0}{2}u^2,
    \end{equation}
	where $F=F(u)$ denotes the Flory-Huggins  entropy of mixing
	\begin{align}
    \label{FL-pot}
		F(u) = \frac{\theta}{2} \bigg((1+u) \ln (1+u) + (1-u) \ln (1-u)\bigg).
	\end{align}
The parameters $\theta$ and $\theta_0$ satisfy the thermodynamic conditions $0< \theta < \theta_0$. We denote by $f$ the first derivative of $F$. 
   
The free energy associated to the system \eqref{A1}--\eqref{A2} is the Flory-Huggins-de Gennes functional
\begin{equation}
\label{FHG}
 E(u) :=\int_{\Omega} \left( \frac{a(u)}{2} |\nabla u|^2 +   \Psi(u) \right) \ \dx.
 \end{equation} 
This is a generalization of the classical Ginzburg-Landau functional, which corresponds to the case $a \equiv 1$. Based on this definition, the function $\mu$ in \eqref{A2} is the chemical potential, defined as the first variational derivative of the free energy $E(u)$.

We equip the system with the following boundary conditions
\begin{equation} \label{boundary conditions}
        \partial_{\mathbf{n}} u=0, \quad b(u)\partial_{\mathbf{n}}\mu = 0 \quad \text{on } \partial \Omega\times (0,T), 
\end{equation}
   where $\mathbf{n}$ is the unit outward normal vector on $\partial \Omega$ and $\partial_{\mathbf{n}}$ denotes the outer normal derivative on $\partial \Omega$,
and the initial condition
\begin{equation}
 \label{initial conditions}
u( \cdot, 0) = u_0(\cdot)  \quad \text{in } \Omega.
\end{equation}  

The system \eqref{A1}--\eqref{A2}, subject to \eqref{boundary conditions}--\eqref{initial conditions}, is characterized by two fundamental physical laws:
any sufficiently smooth solution $u$ satisfies
\begin{enumerate}[i)]
    \item \textit{Mass conservation}:
\begin{equation}
    \label{mass conservation}
    \int_{\Omega} u( \cdot,t) \ \dx = \int_{\Omega} u_{0}(\cdot) \ \dx, \hspace{1.0 cm}\forall \, t \ge 0.
\end{equation}
\item \textit{Energy balance}:
\begin{equation} 
\label{energy balance}
E(u(t)) 
+  \int_0^t \int_\Omega  b(u)|\nabla \mu|^2\ \dx \ds  = 
E(u_0), \hspace{1.0 cm}\forall \, t \ge 0.
\end{equation} 
\end{enumerate}
    
\medskip   
  
Model \eqref{A1}--\eqref{A2} belongs to the class of macroscopic PDEs arising from the Phase Field (or Diffuse Interface) theory, which describes the dynamics of phase separation. In a binary mixture, this phenomenon consists of the segregation from an initially homogeneous (mixed) state, which is thermodynamically unstable, into a demixed pattern characterized by the formation of spatial domains, each predominantly occupied by one of the two components or phases. As the evolution proceeds, these domains merge and grow, giving rise to the coarsening process. 

The classical model \eqref{A1}--\eqref{A2} with $a( \cdot )\equiv 1$ was introduced by Cahn and Hilliard in \cite{CAHN-HILLIARD} (see also \cite{CAHN, HILLIARD}) to study spinodal decomposition and nucleation in alloys or simple liquids. Subsequently, system \eqref{A1}--\eqref{A2} with  singular diffusion $a(u)$ given by \eqref{a} was proposed by de Gennes in \cite{GenJCP} to describe the concentration fluctuations in a binary polymer melts.
In particular, the expression \eqref{a} stems from the assumption that chain connectivity of polymer molecules gives rise to an explicit entropic contribution to the square-gradient coefficient. 

Since then, coarsening dynamics of these systems, namely the time evolution of the characteristic size of the phase domains, has attracted considerable interest; see, for instance, \cite{BINDER, CG1995, CTGM1990, GLOTZER, KO, OE}.

\subsection{Literature review}
The mathematical theory of the Cahn-Hilliard equation \eqref{A1}--\eqref{A2}, supplemented with \eqref{boundary conditions}--\eqref{initial conditions}, strongly depends on the choice of the mobility 
$b$ and the nonlinear diffusion coefficient $a$.

In the classical case of constant mobility and diffusion coefficients (e.g.
$b \equiv 1$ and $a\equiv1$), the theory is by now well established. Global well-posedness of weak solutions, propagation of regularity and related analytical properties have been extensively investigated, for example, in \cite{AbelsARMA,AbWNA,DDNA,EL1991,GGH2018}. Further qualitative properties include the separation property \cite{GGG2023,GP2024,HW2021,MZ2004}, convergence to stationary states  and the existence of global attractors \cite{AbWNA,MZ2004}.

When the mobility is non-degenerate, that is, $b(u)$ is strictly positive in $[-1,1]$, the existence of global weak solutions and the local existence of strong solutions were first established in \cite{BB1999}, while the validity of the energy equality and the existence of global attractors were studied in \cite{S2007}. More recently, uniqueness of weak solutions, global propagation of regularity and convergence to a stationary state were proved in two space dimensions in \cite{CGGG2025}. In three dimensions, the convergence of weak solutions to a stationary state (in the energy norm) was obtained in \cite{GP2025}. 

Another natural generalization consists in considering a variable diffusion coefficient $a(u)$, as already suggested in \cite{CAHN}. Assuming that $a(u)$ is strictly positive and bounded in $[-1,1]$, Schimperna and Paw\l ow in \cite{SP2011} established the existence of global weak solutions. Under an additional convexity assumption on the interfacial energy (see condition (6.1) in \cite[Theorem 6.1]{SP2011}), they also proved uniqueness and global propagation of regularity in three dimensions.
More recently, these results were extended in \cite{CGGS2025}, where the convexity assumption was removed and non-degenerate mobilities were included. In particular, the authors proved uniqueness of weak solutions, global propagation of regularity and convergence to equilibrium in two dimensions, together with the existence of global strong solutions in three dimensions for initial data sufficiently close to energy minimizers (Lyapunov stability). We also mention the recent work \cite{GP2026}, which establishes convergence of weak solutions to a stationary state in three dimensions.

The case of singular diffusion, corresponding to the coefficient $a(u)$ as in \eqref{a}, was investigated by Schimperna and Paw\l ow in \cite{SP2013} 
under the assumption of constant mobility and convex three-dimensional domains. They proved the existence of suitable weak solutions (see \cite[Definition 2.3]{SP2013}), which become classical on every interval $(\tau, T)$, with $0<\tau<T$. Consequently, the equations are satisfied almost everywhere in $\Omega \times (0,T)$, and solutions satisfy the separation property for all positive times. Moreover, they also showed that uniqueness holds in the class of weak solutions that are classical for strictly positive times.

In contrast with the settings discussed above, considerably less is known when the mobility is degenerate. In this case, the only result currently 
available is the existence of global weak solutions. This was first established by Elliott and Garcke \cite{EG1996} for system \eqref{A1}--\eqref{A2}, with mobility given by \eqref{b} (or suitable generalizations) and constant diffusion coefficient $a\equiv 1$. Related existence results were obtained in \cite{YIN} for the one dimensional case and in \cite{DD2016} for the case of smooth double-well potentials $\Psi$. A different approach, based on the interpretation of \eqref{A1}--\eqref{A2} as a gradient flow in weighted Wasserstein spaces, was proposed in \cite{LMS2012}; see also \cite{CMN2019} for a related system. Existence of weak solutions has also been established for a Cahn-Hilliard model coupled with hydrodynamics and endowed with a strictly positive and bounded diffusion coefficient $a(u)$ in \cite{ADG2013-2}.

To the best of our knowledge, the only contribution addressing both degenerate mobility and a singular nonlinear diffusion coefficient is due to Canc{\`e}s and Matthes \cite{CM2023}. There, the authors proved the existence of weak solutions for a nonlocal system of two coupled diffusion equations by means of a gradient-flow approach. We refer the reader to Section 1.1 in \cite{CM2023} for a detailed comparison between that nonlocal model and the local Cahn-Hilliard equation originally proposed by de Gennes in \cite{GenJCP}, which is considered in the present paper.

Finally, we mention that numerical approximation schemes and simulations for the Cahn-Hilliard equation associated with the Flory-Huggins-de Gennes free energy \eqref{FHG} have recently been investigated in \cite{DWWZ2021,DWWZ2022,ZY2020,YANG2016}.

Before presenting the main results of this paper, we briefly discuss the main difficulties arising in the mathematical analysis of the Cahn-Hilliard equation with degenerate mobility. As mentioned above, the currently available theory is limited to the existence of global weak solutions, whose construction requires different arguments compared with the cases of constant or non-degenerate mobility. In particular, the definition of a suitable notion of weak solution, as introduced in \cite{EG1996}, relies on a set of {\it a priori} estimates that can be derived for smooth solutions to \eqref{A1}--\eqref{A2} with $a \equiv 1$:
\begin{itemize}
\item The energy balance \eqref{energy balance} guarantees that the free energy is a Lyapunov functional over the phase space $V=\lbrace u \in H^1(\Omega): \ | u| \leq 1 \text{ a.e. in } \Omega \rbrace$, which translates into a global control of the free energy for all time, provided that the initial free energy $E(u_0)$ is finite.

\item The entropy estimate, namely a differential inequality for the functional
$$
\int_\Omega \Phi(u) \ \dx, \quad \text{where} \quad \Phi^{\prime \prime}(u)= \dfrac{1}{b(u)},
$$
 entails a control of $u$ in $L^2(0,T; H^2(\Omega))$ provided that $\displaystyle \int_\Omega \Phi(u_0) \ \dx $ is finite.
\end{itemize}
These {\it a priori} estimates are sufficient to define weak solutions. Indeed, owing to the relation $b(u) f'(u)= \theta$, the Cahn-Hilliard equation can be equivalently rewritten as 
\begin{equation}
\label{weak-CH}
\partial_t u= - \div \left( b(u) \nabla \Delta u \right) + \theta \Delta u - \theta_0 \div \left( b(u) \nabla u \right)\!,
\end{equation}
which can be interpreted in a variational sense for a function $u\in L^\infty(0,T; H^1(\Omega)) \cap L^2(0,T;H^2(\Omega))$ satisfying $|u(x,t)|\leq 1$ almost everywhere in $\Omega_T$. Furthermore, such estimates are also {\it stable}, namely they hold within a suitable approximation scheme and provide the sufficient compactness to obtain a limit function satisfying such notion of solution. In particular, the approach in \cite{EG1996} follows by approximating the degenerate mobility with positive ones, and the potential $\Psi$ with quadratic polynomials, whereas in \cite{LMS2012} a variational minimizing movement/JKO scheme is employed. 

It is important to emphasize that no {\it a priori} estimates are known for $\Psi'(u)$, and consequently for the chemical potential $\mu$, in any Lebesgue space. Therefore, these quantities do not explicitly appear in the weak formulation of \eqref{weak-CH}. Furthermore, for weak solutions, the flux $b(u)\nabla \mu$ is not defined in the classical sense and must be interpreted through a suitable weak formulation (see the term $\boldsymbol{J}$ in \cite[Theorem 1]{EG1996}). 
 
Finally, the uniqueness of weak solutions, the propagation of regularity, and the characterization of the long-time dynamics remain challenging open problems, even in the case of a constant diffusion coefficient $a$.

\subsection{The main results} 
The aim of this paper is to demonstrate the existence of global weak solutions to system \eqref{A1}--\eqref{A2}, subject to \eqref{boundary conditions}--\eqref{initial conditions}, under the thermodynamically relevant assumptions \eqref{a}--\eqref{FL-pot}.
The main contribution of the paper reads as follows. We refer to Section \ref{Func-set} for the notation.

	\begin{theorem} 
    \label{Main Result}
		Let $T>0$ be chosen arbitrarily and let $\Omega \subset \mathbbm{R}^3$ be a bounded smooth domain. 
		Assume that the initial condition $u_0$ satisfies
        \begin{equation}
            \label{I1}
            u_0 \in H^1(\Omega) \cap L^\infty(\Omega) \quad \text{such that}\quad \|u_0\|_{L^\infty(\Omega)}\le 1,
        \end{equation}
		with 
		\begin{equation}
        \label{mass}
			\overline{u_0}:= \dfrac{1}{|\Omega|}\int_{\Omega} u_0\ \dx \in (-1,1)
		\end{equation}
		 and
		\begin{equation}
		\label{E-0}
 		 E(u_0)=\int_{\Omega} \left( \frac{a(u_0)}{2} |\nabla u_0|^2 +   \Psi(u_0) \right) \ \dx<\infty.
 \end{equation}
		Then, the Cahn-Hilliard system \eqref{A1}--\eqref{A2}, subject to \eqref{boundary conditions}--\eqref{initial conditions}, under the assumptions \eqref{a}--\eqref{FL-pot}, admits a global weak solution $u: \Omega_T  \to \mathbbm{R}$ in the following sense:
		\begin{itemize}
			\item[(i)]  \textbf{Regularity class}: the weak solution $u$ satisfies
            $$-1\le u \le 1 \ \mbox{a.e. in } \Omega_T,$$
            and
            $$u \in L^\infty(0,T;H^1(\Omega))\cap 
             L^2(0,T;H^2(\Omega)) \cap C([0,T]; L^2(\Omega)),$$
            $$\partial_t u  \in L^2(0,T; (H^{1}(\Omega))').$$
			Besides, setting
            $$\phi(s) = \int_0^s \sqrt{\dfrac{a(\tau)}{2}} \mathrm{d}\tau =
            \arcsin (s), \quad \forall \, s \in [-1,1],
            $$
            we have 
            $$ \phi(u) \in L^\infty(0,T;H^1(\Omega))\cap 
            L^2(0,T;H^2(\Omega)).
            $$
            In addition, if $\Omega$ is also convex, then
            $$          
              u\in L^4(0,T;H^2(\Omega))\quad \text{and} \quad \phi(u) \in L^4(0,T;H^2(\Omega)).
            $$
			
			\item[(ii)] \textbf{Weak formulation}: the integral identity
            \begin{equation}
                \label{weak formulation}
                \int_0^T \langle \partial_t u, v \rangle \  \dt = -\int_{\Omega_T} \boldsymbol{J} \cdot \nabla v \ \dx \dt,
            \end{equation}
            holds for any $v \in L^2(0,T;H^1(\Omega))$,
            where $\boldsymbol{J} \in L^2(0,T;L^2(\Omega))$, satisfies 
			\begin{align}
			\int_{\Omega_T} \boldsymbol{J} \cdot \boldsymbol{\eta} \ \dx \dt &= \int_{\Omega_T} \left[2\sqrt{b(u)}\Delta \phi(u) \right]\div\,  \boldsymbol{\eta} \ \dx \dt  \notag 
\\[3pt]
& \quad + \int_{\Omega_T}  \left[ 2 b'(u) \Delta \phi(u)\nabla \phi(u) 
+ (b\Psi'')(u) \nabla u \right] \cdot \boldsymbol{\eta} \ \dx \dt, 
\label{J}
			\end{align}
for all $\boldsymbol{\eta} \in C_c^\infty\left( \overline{\Omega} \times (0,T); \mathbbm{R}^3\right)$ with $\boldsymbol{\eta} \cdot \mathbf{n}=0$ on $\partial \Omega \times(0,T)$, where
$$
(b\Psi'')(s)= \theta - \theta_0 b(s), \quad \forall \, s \in [-1,1].
$$		
Besides, $u(\cdot, 0)=u_0(\cdot)$ almost everywhere in $\Omega$. 

        \vspace{0.05cm}
	
\item[(iii)] \textbf{Energy inequality}: there exists $\widehat{\boldsymbol{J}} \in  L^2(0,T;L^2(\Omega))$ such that $\boldsymbol{J} = \sqrt{b(u)} \widehat{\boldsymbol{J}} $ and
			the following integral inequality
			\begin{equation}
			\label{energy-inequality}
			E(u(t)) 
+  \int_0^t \int_\Omega  |\widehat{\boldsymbol{J}}|^2 \ \dx \ds  \leq
E(u_0)
\end{equation}			 
holds for almost any $t \in (0,T)$.   
		\end{itemize}
	\end{theorem}

\begin{remark}
A few observations are in order: 

\begin{itemize}
\item The assumption \eqref{mass} ensures that the initial datum is not a pure state (i.e. $u\equiv 1$ or $u \equiv -1$) and it is required to carry out our suitable approximation procedure. On the other hand, we observe that if $\overline{u_0}=1$ (respectively, $\overline{u_0}=-1$), then the global weak solution can be simply $u(t)\equiv1$ (respectively, $u(t)\equiv-1$), namely no phase separation occurs, which complies with the statement provided in Theorem \ref{Main Result}.

\item In Theorem \ref{Main Result}, we establish the existence of global weak solutions on the interval $[0,T]$, where $T>0$ is arbitrary. Since the approximating solutions $u_n$ constructed in Section \ref{Approximation} exist on the whole interval $[0,\infty)$ and satisfy the estimates derived in Section \ref{s-proof-main}, a standard diagonal argument by successive extractions yields the existence of a weak solution on the entire time interval $[0,\infty)$.
\end{itemize}
\end{remark}

Next, we show that every weak solution given by Theorem \ref{Main Result} is also a weak solution in the sense of Canc{\`e}s and Matthes in \cite{CM2023}. More precisely, the following result holds.

\begin{theorem}
\label{formulazione alla Cances}
Let $u$ be a solution in the sense of Theorem \ref{Main Result}. Then, setting
\begin{equation}
    \label{q}
    q:= -2\Delta \phi(u) + \sqrt{b(u)}\Psi'(u),
\end{equation}
the integral identity
    \begin{equation}
        \label{weak formulation Cances}
        \int_0^T \langle \partial_t u, \eta \rangle \  \dt = 
         \int_{\Omega_T} q\left[\sqrt{b(u)} \Delta \eta + b'(u)\nabla \phi(u) \cdot \nabla \eta \right]\dx \dt
    \end{equation}
    holds for any $\eta \in C_c^\infty\left( \overline{\Omega} \times (0,T)\right)$ such that $\nabla_\mathbf{n} \eta=0$ on $\partial \Omega \times (0,T)$. 
\end{theorem}

\subsection{Key ideas}  To handle the singular diffusion term in \eqref{A1}--\eqref{A2}, we introduce the auxiliary variable
$$
\phi(u)=\arcsin (u).
$$
This choice is naturally motivated by the free-energy density \eqref{FHG}, since, at least formally,
$$
\frac{a(u)}{2} |\nabla u|^2 = \frac{1}{1-u^2} |\nabla u|^2 =
|\nabla \phi(u)|^2.
$$
In addition, \eqref{A2} can be rewritten as
\begin{equation}
\label{mu-phi}
\mu= - 2 \phi'(u) \Delta \phi(u) + f(u) -\theta_0 u.
\end{equation}
By analogy with the classical Cahn-Hilliard equation with degenerate mobility (i.e., when $a(u)\equiv 1$), one cannot expect the chemical potential $\mu$ to be well defined by \eqref{mu-phi}. Therefore, the degeneracy of the  mobility must be exploited to compensate the singular terms in \eqref{mu-phi}.
To this end, a key observation underlying our analysis is that the degenerate mobility $b$ is related both to the free-energy density $F$ and to the function $a$ through the key identities
\begin{equation}
\label{key relation}
      a(u)b(u)=2, \quad b(u) f'(u)= \theta.
\end{equation} 
We now summarize the main ideas of our proofs:
\begin{itemize}

\item[1.] We introduce a suitable notion of weak solution (see \eqref{weak formulation}--\eqref{J}). This notion is inspired by the work by Elliott and Garcke \cite{EG1996}. To motivate it, let $u$ be a smooth solution to \eqref{A1}--\eqref{A2} and let $\boldsymbol{\eta} \in C_c^\infty\left( \overline{\Omega} \times (0,T); \mathbbm{R}^3\right)$ be a test function with $\boldsymbol{\eta} \cdot \mathbf{n}=0$ on $\partial \Omega \times(0,T)$.
While the weak formulation \eqref{weak formulation} is standard, the variational formulation of the flux $\boldsymbol{J}$ arises from \eqref{mu-phi} and the following formal integration by parts: 
\begin{align}
\int_{\Omega_T} b(u) \nabla \mu \cdot \boldsymbol{\eta}  \ \dx\dt
&= \int_{\Omega_T}  \nabla \left( b(u)\mu\right) \cdot \boldsymbol{\eta}  \ \dx\dt
-\int_{\Omega_T}  \mu \nabla b(u) \cdot \boldsymbol{\eta}  \ \dx\dt
\notag
\\
&=
- \int_{\Omega_T}  b(u)\mu  \div \, \boldsymbol{\eta}  \ \dx\dt
-\int_{\Omega_T}  \mu b'(u) \nabla u \cdot \boldsymbol{\eta}  \ \dx\dt
\notag
\\
&=
 \int_{\Omega_T}  \left( 2 b(u) \phi'(u) \Delta \phi(u) - b(u)\Psi'(u) \right)  \div \, \boldsymbol{\eta}  \ \dx\dt 
 \notag
 \\
 &\quad
 +\int_{\Omega_T}  \left(  2 \phi'(u) \Delta \phi(u) b'(u) \nabla u 
 - b'(u) \Psi'(u) \nabla u \right)  \cdot \boldsymbol{\eta}  \ \dx\dt 
 \notag
 \\
 &=
 \int_{\Omega_T}  2 \sqrt{b(u)} \Delta \phi(u)   \div \, \boldsymbol{\eta}  \, \dx\dt
 +\int_{\Omega_T}   2 \Delta \phi(u) b'(u) \nabla \phi(u) \cdot \boldsymbol{\eta} \ \dx\dt  
 \notag
 \\
 &\quad
 + \int_{\Omega_T} b(u) \Psi''(u) \nabla u  \cdot \boldsymbol{\eta}  \ \dx\dt.
 \notag
\end{align}

\item[2.] We construct an approximation scheme 
by replacing the degenerate mobility $b$ by a sequence of strictly positive mobilities $b_\ve$, and the singular diffusion coefficient $a$ with a sequence of regularized coefficients $a_\ve$. This is achieved in \eqref{a-b-def} in such a way that the original relation  \eqref{key relation} is preserved for $b_\ve$ and $a_\ve$. A distinctive feature is that the potential $\Psi$ keeps the original form \eqref{Psi}--\eqref{FL-pot}. This allows us to apply \cite[Theorem 1.3]{CGGS2025} to obtain a family of approximating global weak solutions $u_\ve$ to the regularized problem, where $b$ and $a$ are replaced by $b_\ve$ and $a_\ve$, respectively. In particular, the constructed solutions $u_{\ve}$ satisfy the physical constraint $|u_{\ve}|<1$ almost everywhere in $\Omega_T$ (unlike the approximation considered in \cite{EG1996}). The latter immediately entails that the limit function $u$ satisfies the physically relevant condition
$$
 -1 \leq u \leq 1  \quad \mbox{a.e. in } \Omega_T.
 $$ 

\item[3.] The crucial part of the proof consists in deriving uniform bounds for the approximating solutions $u_\ve$ as the parameter $\varepsilon \to 0$. 
First, as a consequence of the energy inequality for $u_\ve$ (cf. \eqref{energy inequality}), we obtain a uniform bound of $\nabla \phi_\ve (u_\ve)$ in $L^\infty(0,T; L^2(\Omega))$ as well as of $\sqrt{b_\ve(u_\ve)} \nabla \mu_\ve$ in $L^2(0,T; L^2(\Omega))$. We point out that the former provides a richer information than a uniform bound of $\nabla u_\ve$ in $L^\infty(0,T; L^2(\Omega))$, available in the case $a\equiv 1$.
However, these estimates are not sufficient to obtain the compactness required to identify the limit of the terms involving the singular diffusion coefficient.

The core argument is therefore the derivation of a uniform estimate of $\Delta \phi_\ve(u_\ve)$, which in turn yields a control on $\Delta u_\ve$.
We first consider the case where the domain $\Omega$ is convex. In this setting, our strategy relies on an elliptic-type estimate that allows us to fully exploit the first energy bounds, in particular the uniform control of $\nabla\phi_\ve(u_\ve)$ in $L^\infty(0,T;L^2(\Omega))$. Starting from relation \eqref{SE 1}, the key point is to extract from the first term on the left-hand side a coercive contribution controlling $\|\Delta\phi_\ve(u_\ve)\|_{L^2(\Omega)}^2$. To this end, we employ a general result (see Lemma \ref{refined-lemma}) inspired by the work of Canc{\`e}s and Matthes in \cite{CM2023}, for which the convexity of $\Omega$ plays a crucial role. This eventually yields uniform bounds for both $\phi_\ve(u_\ve)$ and $u_\ve$ in $L^4(0,T;H^2(\Omega))$.

 We then turn to general domains, for which the convexity-based argument is no longer available. We therefore resort to the so-called entropy method (see \cite{ADG2013-2, CM2023, EG1996,LMS2012}) and we perform an entropy-like estimate that provides a uniform control of
  $$
  \sqrt{\phi_\ve'(u_\ve)}\,\Delta\phi_\ve(u_\ve)
  \quad\text{in }L^2(Q_T),
  $$
 which allows to recover the uniform control of both $\phi_\ve(u_\ve)$ and $u_\ve$ in $L^2(0,T;H^2(\Omega))$. This approach applies without any convexity assumption on $\Omega$, at the price of a weaker time-integrability estimate. Nonetheless, such uniform estimates of the second-order spatial derivatives are still sufficient to provide the
suitable compactness framework for the construction of weak solutions.

We highlight that, in comparison with the existing
literature on free energy with singular diffusion \cite{CM2023, SP2013}, this is the first existence result that
does not require the convexity of the domain.

\item[4.]  The final step in the existence proof consists in passing to the limit as $\ve \to 0$ and showing that the limit function $u$ is a weak solution in the sense of Theorem \ref{Main Result}.
In this step, due to the particular choice of $b_\varepsilon$ (cf. \eqref{a-b-def} and \eqref{Cancellazione}), we need to prove that the integral (see \eqref{limitequi})
\begin{equation}
\label{to0}
\varepsilon \int_{\Omega_T} f(u_\ve) \, \div \boldsymbol{\eta} \ \dx \dt
\quad \text{ converges to } 0.
\end{equation}
To this end, we establish a bound for $f(u_\ve)$ by adapting the classic tool from \cite{MZ2004} and keeping track of the dependence on the approximating parameter. The resulting estimate is given in \eqref{f L1} below, which implies that \eqref{to0} holds.

\item[5.] Our final result establishes that every solution in the sense of Theorem \ref{Main Result} also satisfies the weak formulation \eqref{weak formulation Cances}, inspired by that introduced by Cancès and Matthes in \cite{CM2023}. In this formulation, the flux $\boldsymbol{J}$ is replaced by the weighted chemical potential $q$, formally defined as $\sqrt{b(u)}\mu$. A key advantage of this approach is that $q$ is defined almost everywhere through the identity \eqref{q}.
The proof relies on a characterization of the gradient of $G(u)$ (see \eqref{nabla G}), where $G$ denotes the antiderivative of $b'(s)f(s)$.
\end{itemize}

We expect that the techniques developed in this work, in particular the new estimates providing control of the second-order spatial derivatives of $u$ and $\phi(u)$, can be adapted to a broader class of mobility functions, including $b(u)=(1-u^2)^m$ with $m>1$, which have been considered in \cite{EG1996, LMS2012}. In addition, it would be interesting to extend the present analysis to Navier-Stokes/Cahn-Hilliard systems and to multicomponent versions of \eqref{A1}--\eqref{A2}. We plan to address these problems in future works.

\section{Functional setting}
\label{Func-set}
Let $X$ be a real Banach space. The set $X'$ denotes its dual space and $\langle\cdot, \cdot \rangle_{X',X}$ the duality product. For $p\in[1,\infty]$ and an interval $I \subseteq [0,\infty)$, the Lebesgue space $L^p(I;X)$ is the set of all strongly measurable functions $p$-integrable/essentially bounded from $I$ to $X$. The space $L^p_{\rm uloc}([0,\infty); X)$ consists of all functions $f \in L^p(0,T;X)$ for any $T>0$ such that
$$\|f\|_{L^p_{\rm uloc}([0,\infty); X)}:= \sup_{t\ge 0} \left( \int_t^{t+1} 
\|f(s)\|^p_X \ds \right)^\frac{1}{p} <\infty.$$
The set of continuous functions $f: [0,T]\to X$ is denoted by $C([0,T];X)$, endowed with the supremum norm. 

Let $\Omega $ be a bounded domain in $\mathbbm{R}^{3}$. 
The Sobolev spaces of functions $f:\Omega \rightarrow \mathbbm{R}$ in $L^p(\Omega)$, with distributional derivative of order up to $k$ in $L^p(\Omega)$, are denoted by $W^{k,p}(\Omega )$ where $k\in \mathbbm{N}$ and $1\leq p\leq \infty $. The notation $\Vert \cdot \Vert
_{W^{k,p}(\Omega )}$ represents its norm. In particular, we use the notation $H^k(\Omega)$ for the Hilbert space with $p=2$. For simplicity, we will denote by $\langle\cdot, \cdot \rangle$ the duality product between $(H^1(\Omega))'$ and $H^1(\Omega)$.

We recall the classical Poincar\'{e}-Wirtinger inequality in bounded Lipschitz domains $\Omega \subset \mathbbm{R}^3$
\begin{equation}
    \label{PI}
    \|f \|_{L^2(\Omega)} \le C \left( \|\nabla f\|_{L^2(\Omega)} + |\overline{f}| \right)\!, \quad \forall \, f \in H^1(\Omega),
\end{equation}
where $C=C(\Omega)$ is positive constant and $\overline{f}$ denotes the total mass (spatial average) of an integrable function defined as
$$
\overline{f}= \frac{1}{|\Omega|} \int_\Omega f \ \dx, \quad \forall \, f \in L^1(\Omega).
$$
We will also use the following Gagliardo-Nirenberg interpolation inequalities in Sobolev spaces 
\begin{equation}
    \label{GNI L4}
    \|\nabla f \|_{L^4(\Omega)}\le C \|f\|_{L^\infty(\Omega)}^{\frac12}\|f\|_{H^2(\Omega)}^{\frac12}, \quad \forall \, f \in H^2(\Omega)
\end{equation}
and
\begin{equation}
    \label{GNI L3}
    \|\nabla f \|_{L^3(\Omega)}\le C \|f\|_{L^6(\Omega)}^{\frac12}\|f\|_{H^2(\Omega)}^{\frac12},  \quad \forall \, f \in H^2(\Omega).
\end{equation} 
Besides, we report the following version of \cite[Lemma 5.1]{GST2009}.
\begin{lemma}
\label{lemma H2}
    Let $\Omega$ be a smooth convex set. Assume that
    $f \in W^{2,2}(\Omega)$ with $\partial_{\mathbf{n}} f=0$ on $\partial \Omega$. Then
    $$
    \int_\Omega |\Delta f|^2 \ \dx \ge \int_\Omega |D^2 f |_F^2 \ \dx,
    $$
where $D^2 f$ is the Hessian of $f$ and $|A|_F= \sqrt{\textrm{tr}(A^TA)}$ is the Frobenius norm of a square matrix $A$. 
\end{lemma}

In the sequel, we set $\Omega_T= \Omega \times (0,T)$ and we denote with
\begin{align*}
    & f_\ve \to f \quad \mbox{a.e. in } \Omega_T \mbox{ the almost everywhere convergence in } \Omega_T \mbox{ of a sequence }\{f_\ve\}_\ve \mbox{ to } f \mbox{ as } \ve \to 0,
    \\[3pt]
    & f_\ve \to f \quad \mbox{in }X \mbox{ the strong convergence of a sequence }\{f_\ve\}_\ve \mbox{ to } f \mbox{ in } X \mbox{ as } \ve \to 0,\\[3pt]
     & f_\ve \rightharpoonup f \quad \mbox{in }X \mbox{ the weak convergence of a sequence }\{f_\ve\}_\ve \mbox{ to } f \mbox{ in } X \mbox{ as } \ve \to 0,\\[3pt]
      & f_\ve  \overset{\star}{\rightharpoonup} f \quad \mbox{in }X \mbox{ the weak-$\star$ convergence of a sequence }\{f_\ve\}_\ve \mbox{ to } f \mbox{ in } X \mbox{ as } \ve \to 0.
\end{align*}
Lastly, we report a standard convergence result, that is an application of the Vitali's lemma (see, e.g., \cite{E1990}). 
\begin{lemma}\label{convergenza alla Vitali}
 Let $\Omega \subset \mathbbm{R}^3$ be bounded and $f_n: \Omega_T  \to \mathbbm{R}$ be a sequence in $L^1( \Omega_T )$. Suppose 
    \begin{enumerate}
        \item $f_n \to f$ a.e. in $\Omega_T$,
        \item the sequence $f_n$ is bounded in $L^p( \Omega_T )$ for some $p>1$.
    \end{enumerate}
    Then,
    $$
    f_n \to f \quad \mbox{in }L^r( \Omega_T), \quad \text{ for all } \, 1 \le r <p.
    $$
\end{lemma}

\section{Approximation scheme}
\label{Approximation}
In this section, we introduce the approximation scheme employed to prove Theorem \ref{Main Result}. For any given parameter $\varepsilon>0$, we approximate the degenerate mobility $b$ by a sequence of strictly positive functions $\lbrace b_\ve \rbrace$. We then approximate the singular coefficient $a$ by a corresponding sequence $\lbrace a_\ve \rbrace$, defined through the fundamental relation between $a$ and $b$ given by \eqref{key relation}.
    
 Next, we exploit the result proved in \cite{CGGS2025} (see also \cite{SP2011}), which guarantees the existence of weak solutions to system \eqref{A1}--\eqref{A2} subject to \eqref{boundary conditions}--\eqref{initial conditions}, with $b$ and $a$ replaced by $b_\ve$ and $a_\ve$, respectively. 
	
\subsection{Approximation of mobility and diffusion functions}
For any $0<\ve<1 $, we consider the following approximations of the degenerate mobility $b$ and the singular diffusion coefficient $a$
\begin{equation}
\label{a-b-def}
b_{\ve}(s) = 1-s^2+\ve, 
\qquad
a_{\ve}(s) = \dfrac{2}{b_{\ve}(s)}=\dfrac{2}{1-s^2+\ve}, \qquad \forall \, s \in [-1,1].
\end{equation}
Clearly, for any $\ve \in (0,1)$, $b_{\ve}: [-1,1]\rightarrow (0,\infty)$ satisfies
$$
b_{\ve} \in C^1([-1,1]): \quad0<\ve \le b_{\ve}(s) \leq 2, \quad \forall \, s \in [-1,1],
$$
and $a_{\ve}:[-1,1] \rightarrow(0,\infty)$ is such that
$$
a_{\ve} \in C^2([-1,1]): \quad 0< 1 \le a_{\varepsilon}(s) \le \frac{2}{\varepsilon},\quad  \forall \, s \in [-1,1].
$$

\subsection{Existence of approximating solutions}
For any fixed $\ve \in (0,1)$, \cite[Theorem 1.3]{CGGS2025} provides the existence of weak solutions for any finite energy initial data; more precisely, the following result holds. 
	
	\begin{theorem} \label{Ex Weak App}
		Let $\Omega \subset \mathbbm{R}^3$ be a bounded smooth domain. Let $u_0$ satisfy conditions \eqref{I1}--\eqref{E-0}. 
		Then, for any fixed $\varepsilon \in (0,1)$, the Cahn-Hilliard system \eqref{A1}--\eqref{boundary conditions} with $b$ and $a$ replaced by $b_{\ve}$ and $a_{\ve}$, respectively,
		admits a global weak solution $u_{\ve}: \Omega_T \to \mathbbm{R}$ in the following sense:
		\begin{itemize}
			\item[(i)]  \textbf{Regularity and weak formulation}: 
			\begin{subequations} \label{B1}
				\begin{align}
					&u_{\ve} \in L^\infty(0,\infty; H^1(\Omega))\cap L^4_{\rm uloc}([0,\infty); H^2(\Omega))\cap L^2_{\rm uloc}([0,\infty); W^{2,6}(\Omega)), \label{B1.1}\\
					&u_{\ve} \in L^\infty(\Omega \times (0,\infty)) \textrm{ such that } |u_{\ve}(x,t)|<1 \textrm{ a.e. in }\Omega \times (0,\infty),  \label{B1.2}\\
					 &\partial _t u_{\ve} \in L^2(0,\infty; (H^{1}(\Omega))'),\\
                     &\mu_{\ve} \in L^2_{\rm uloc}([0,\infty); H^{1}(\Omega)), \hspace{1.0 cm}f(u_{\ve}) \in L^2_{\rm uloc}([0,\infty); L^{6}(\Omega)),
                     \label{B1.3}
				\end{align}
			\end{subequations}
 and the following relation holds
           \begin{equation}
               \label{ueps}
               \langle \partial_t u_{\ve},v \rangle + (b_{\ve}(u_{\ve})\nabla \mu_{\ve}, \nabla v)=0, \quad\forall \, v \in H^1(\Omega), \textrm{ a.e. in }(0,\infty),
           \end{equation}
where the chemical potential $\mu_{\ve}$ is given by
            \begin{equation}
            \label{mueps}
                \mu_{\ve}= -\div (a_{\ve}(u_{\ve})\nabla u_{\ve})+ \dfrac{a_{\ve}'(u_{\ve})}{2} |\nabla u_{\ve}|^2+f(u_{\ve})-\theta_0 u_{\ve}, \quad \textrm{a.e. in }\Omega \times (0,\infty).
            \end{equation}
            Furthermore, $\partial_{\mathbf{n}}u_{\ve}=0$ almost everywhere on $\partial \Omega \times (0,\infty)$ and $u_{\ve}(\cdot , 0)= u_0(\cdot)$ in $\Omega$. 
            
            \vspace{0.05cm}
			\item[(ii)] \textbf{Mass conservation}: the total mass of the solution $u_\ve$ is conserved over time, 
            \begin{equation}
                \label{massueps}
                \overline{u_{\ve}}(t)= \dfrac{1}{|\Omega|}\int_{\Omega} u_{\ve}(t) \ \dx = \dfrac{1}{|\Omega|}\int_{\Omega} u_0 \ \dx = \overline{u_0}, \hspace{1.0 cm}\forall \, t \ge 0.
            \end{equation}
            
            \vspace{0.05cm}
			\item[(iii)] \textbf{Energy inequality}: defining, for $ t \ge 0 $,
            $$
            E_{\ve}(u_\ve(t)):= \int_{\Omega} \left(\dfrac{a_{\ve}(u_{\ve})}{2}|\nabla u_{\ve}|^2+ \Psi(u_\ve)\right) \dx,
            $$
			the following integral inequality holds
			\begin{equation} \label{energy inequality}
				E_{\ve}(u_\ve(t)) +  \int_0^t \int_{\Omega} b_{\ve}(u_{\ve})|\nabla \mu_{\ve}|^2 \ \dx \ds \le E_{\ve}(u_0), \quad \forall \, t \geq 0.
			\end{equation} 
		\end{itemize} 
	\end{theorem}

\begin{remark}
We recall that the convexity of $\Omega$ is not required in Theorem \ref{Ex Weak App}.
\end{remark}

	\section{Existence of weak solutions: proof of Theorem \ref{Main Result}} 
	\label{s-proof-main}
	
	Let $T>0$ be chosen arbitrarily and let us consider the family $\{ u_{\ve} \}_{\ve>0}$ of weak solutions, with corresponding chemical potential $\mu_{\ve}$, provided by Theorem \ref{Ex Weak App}.
For any $0< \varepsilon <1$, we introduce the family of functions $\{ \phi_\ve(\cdot) \}_{\ve>0}$, defined as follows
$$
\phi_{\ve}: [-1,1] \to \mathbbm{R}
$$
 \begin{equation}
        \label{phieps}
        \phi_{\ve}(s):= \int_0^{s} \sqrt{\dfrac{a_{\ve}(r)}{2}}\ \dr 
        = \int_0^{s}  \dfrac{1}{\sqrt{1-r^2+\ve}}\ \dr = \arcsin \left( \dfrac{s}{\sqrt{1+\ve}}\right)\!, 
     \quad \forall \, s \in [-1,1].
    \end{equation}
Since the function $s \mapsto \frac{s}{\sqrt{1+\varepsilon}}$ maps $[-1,1]$ onto $\left[ \frac{-1}{\sqrt{1+\varepsilon}}, \frac{1}{\sqrt{1+\varepsilon}}\right]$ and the $\arcsin(\cdot) \in C^\infty(-1,1)$, we deduce that $\phi_{\ve} \in C^\infty([-1,1])$, for any $0<\varepsilon<1$. In particular, we observe that
$$
 \sqrt{a_\ve(s)}= \sqrt{2}\phi'_{\ve}(s), \quad \forall \, s \in [-1,1].
$$
 Recalling that $u_\ve: \Omega_T \to (-1,1)$ for any $ \ve \in (0,1)$ owing to \eqref{B1.2},
we consider the family of functions $\{ \phi_\ve(u_{\ve}) \}_{\ve>0}$ defined as follows
$$
\phi_\ve(u_{\ve})=\arcsin \left( \dfrac{u_{\ve}}{\sqrt{1+\ve}}\right).
$$ 
Then, in light of the regularity properties of $\phi_\ve(\cdot)$ and $u_\varepsilon$ (cf. \eqref{B1.1}), all the subsequent computations are rigorously justified by the standard chain rule for the composition of functions:
 \begin{align}
    \nabla \phi_{\ve}(u_{\ve})&=\phi'_{\ve}(u_{\ve})\nabla u_{\ve} = \sqrt{\dfrac{a_{\ve}(u_{\ve})}{2}} \nabla u_{\ve},\hspace{1.0 cm} \sqrt{a_{\ve}(u_{\ve})}\nabla u_{\ve}= \sqrt{2}\nabla \phi_{\ve}(u_{\ve}),
 \label{nabla phi}
 \end{align}
 and 
 \begin{align}
    \Delta \phi_{\ve} (u_{\ve})&= \phi''_{\ve}(u_{\ve})|\nabla u_{\ve}|^2+\phi'_{\ve}(u_{\ve})\Delta u_{\ve}. \label{delta phi}
    \end{align}
    We can thus reformulate \eqref{mueps}, as in \cite{SP2013}, in terms of $\phi_{\ve}$ as follows
    \begin{align}
    \mu_{\ve} & = - \div(a_{\ve}(u_{\ve}) \nabla  u_{\ve}) +\dfrac{a_{\ve}'(u_{\ve})}{2} |\nabla u_{\ve} |^2 + f(u_{\ve}) -\theta_0 u_{\ve}, \notag\\[3pt]
    &= - a_{\ve}(u_{\ve}) \Delta u_{\ve} -\dfrac{a_{\ve}'(u_{\ve})}{2} |\nabla u_{\ve} |^2 + f(u_{\ve}) -\theta_0 u_{\ve}, \notag\\[3pt]
    &= -\sqrt{a_\ve(u_\ve)}\div(\sqrt{a_\ve(u_\ve)} \nabla u_\ve) + f(u_{\ve}) -\theta_0 u_{\ve} , \notag\\[3pt]\label{newformulation}
    &= -2 \phi_\ve'(u_\ve)\Delta \phi_\ve(u_\ve) + f(u_{\ve}) -\theta_0 u_{\ve}, \hspace{2.0 cm} \textrm{ a.e. in }\Omega_T. 
    \end{align}
     
In order to pass to the limit as $\ve \to 0$, suitable compactness properties are required for the sequences $u_\ve$ and $\phi_\ve(u_\ve)$. To this end, we derive some uniform estimates which are independent of $\ve$. In what follows, $C$ denotes a positive constant possibly depending on the parameters of the system, the initial condition $u_0$ and $T$, but independent of $\ve$.
	
\subsection{Energy estimates}
\label{ss1}
Since $u_0$ satisfies \eqref{I1}--\eqref{E-0} and $a_{\ve}(s) \le a(s)$ for any $s\in (-1,1)$, we deduce that $E_{\ve}(u_0) \leq E(u_0)<\infty$, for any fixed $\ve \in (0,1)$.  Therefore, as a consequence of the energy inequality \eqref{energy inequality}, we obtain the following uniform bounds: 
	\begin{align}
		\| F(u_{\ve}) \|_{L^{\infty}(0,T; L^{1}(\Omega))} &\leq C, \label{E1}\\
        \left\|\sqrt{ a_{\ve}(u_{\ve}) }\nabla u_{\ve} \right\|_{L^{\infty}(0,T; L^2(\Omega))} &\leq C, \label{E2}\\
        \left\|\sqrt{ b_{\ve}(u_{\ve}) }\nabla \mu_{\ve} \right\|_{L^{2}(0,T; L^2(\Omega))} &\leq C. \label{E3}
	\end{align}
    Since $a_{\ve}(s) \ge 1$ for any $\ve \in (0,1)$, \eqref{E2} implies that
    \begin{equation}
        \label{E2bis}
        \| \nabla u_{\ve} \|_{L^{\infty}(0,T; L^{2}(\Omega))} \leq C.
    \end{equation}
The latter, combined with the Poincar\'{e} inequality \eqref{PI} and the mass conservation \eqref{massueps}, leads to the uniform bound
    \begin{equation}
        \label{E2tris}
        \| u_{\ve} \|_{L^{\infty}(0,T; H^{1}(\Omega))} \leq C.
    \end{equation}
 Moreover, recalling \eqref{nabla phi}, we also deduce from \eqref{E2}  that
    \begin{equation}
        \label{E2quattro}
        \| \nabla \phi_{\ve}(u_{\ve}) \|_{L^{\infty}(0,T; L^{2}(\Omega))} \leq C.
    \end{equation}
In addition, noticing that $|\phi_{\ve}(u_{\ve})| \le |\phi(u_{\ve})|= |\arcsin (u_\varepsilon)|$, which is uniformly bounded since \eqref{B1.2} holds, it follows that
	\begin{equation} \label{E2cinque}
		\| \phi_{\ve}(u_{\ve})\|_{L^{\infty}(0,T; H^1(\Omega))} \leq C.
	\end{equation}

\subsection{Key estimate of $\phi_\varepsilon(u_\varepsilon)$}
\label{Novel}

We need some supplementary estimates for $u_\ve$ and $\phi_\ve(u_\ve)$,
which can be achieved by choosing a suitable test function in the equivalent formulation \eqref{newformulation}.
We consider two different cases.

\subsubsection{Convex domains.}

Taking $-\Delta u_\ve = - \div\, \left( \sqrt{b_\ve(u_\ve)}\nabla \phi_\ve(u_\ve) \right)$ as test function in \eqref{newformulation}, we obtain the integral identity
\begin{align} 
		2\int_\Omega &\phi_\ve'(u_\ve) \Delta \phi_\ve(u_\ve) \Delta u_\ve \ \dx - \int_\Omega f(u_\ve) \Delta u_\ve \  \dx \notag \\[3pt]
		&= \int_\Omega \sqrt{b_\ve(u_\ve)}\nabla \phi_\ve(u_\ve) \cdot \nabla \mu_\ve \ \dx + \theta_0 \int_\Omega |\nabla u_\ve |^2 \ \dx.
\label{SE 1}
\end{align}
The main point is the control of the first integral on the left-hand side. 
First, thanks to the identity \eqref{delta phi}, we rewrite 
\begin{equation}
    \label{SE 2}
 2 \int_\Omega \phi_\ve'(u_\ve) \Delta \phi_\ve(u_\ve) \Delta u_\ve \ \dx 
 =
 2 \int_\Omega |\Delta \phi_\ve(u_\ve)|^2 \ \dx - 2 \int_\Omega \dfrac{\phi_\ve''(u_\ve)}{(\phi_\ve'(u_\ve))^2}\Delta \phi_\ve(u_\ve)|\nabla \phi_\ve(u_\ve)|^2 \ \dx.
\end{equation}
In particular, by using definition \eqref{phieps}, we obtain
 \begin{align}
&2\int_\Omega \phi_\ve'(u_\ve)\Delta \phi_\ve(u_\ve)\Delta u_\ve \ \dx  \notag
\\
&\quad = 2 \int_\Omega |\Delta \phi_\ve(u_\ve)|^2 \ \dx -2 \int_\Omega \dfrac{u_\ve}{\sqrt{1- u_\ve^2+\ve}}|\nabla \phi_\ve(u_\ve)|^2 \Delta \phi_\ve(u_\ve)\ \dx.
    \label{riscrivo}
\end{align}
Now, we recall that
$$\phi_\ve(u_\ve)= \arcsin \left( \dfrac{u_\ve}{\sqrt{1+\ve}}\right)\!, \hspace{1.0 cm} u_\ve = \sqrt{1+\ve}\sin(\phi_\ve(u_\ve)).$$
In particular, the definition of $\phi_\ve(u_\ve)$ ensures $\phi_\ve(u_\ve) \in \displaystyle \left( -\dfrac{\pi}{2}, \dfrac{\pi}{2}\right)$, and thus $\cos(\phi_\ve(u_\ve))>0$.
This leads to 
\begin{align*}
- \int_\Omega \dfrac{u_\ve}{\sqrt{1- u_\ve^2+\ve}}|\nabla \phi_\ve(u_\ve)|^2 \Delta \phi_\ve(u_\ve)\ \dx 
& = 
\int_\Omega \dfrac{-\sqrt{1+\ve}\sin(\phi_\ve(u_\ve))}{\sqrt{1- (1+\ve)\sin^2(\phi_\ve(u_\ve))+\ve}}|\nabla \phi_\ve(u_\ve)|^2 \Delta \phi_\ve(u_\ve)\ \dx
\\[3pt] 
& = 
\int_\Omega - \dfrac{\sin(\phi_\ve(u_\ve))}{\cos(\phi_\ve(u_\ve))}|\nabla \phi_\ve(u_\ve)|^2 \Delta \phi_\ve(u_\ve)\ \dx.
\end{align*}
In order to control the term on the right-hand side of \eqref{riscrivo}, we introduce the following result inspired by the work of Canc{\`e}s and Matthes \cite{CM2023}.
\begin{lemma}
\label{refined-lemma}
Let $\Omega\subset\mathbbm{R}^d$ be a bounded, smooth and convex domain with $d=2,3$, and let $v\in H^2(\Omega)$ satisfy $\partial_\mathbf{n} v=0$ on $\partial \Omega$.
Let $J\subset\mathbbm{R}$ be an interval such that $v(x)\in J$ for all
$x\in\overline{\Omega}$, and assume that $H\in W^{2,\infty}(J)$. Set
\[
    I_H
    :=
    \int_\Omega H'(v)|\nabla v|^2\Delta v \ \dx.
\]
Then,
\begin{align}
    I_H+\int_\Omega|\Delta v|^2\ \dx
    \geq{}&
    \frac1d\int_\Omega|\Delta v|^2\ \dx
    \nonumber\\
    &+
    \frac{d}{d+2}
    \int_\Omega
    \left[
        -H''(v)
        -
        \frac{d}{d+2}|H'(v)|^2
    \right]
    |\nabla v|^4\ \dx.
\label{eq:refined-lemma}
\end{align}
\end{lemma}

\begin{proof}
By density it is sufficient to assume that $v \in H^3(\Omega)$ with 
$v(x)\in J$ for all $x\in\overline{\Omega}$ and
$\partial_{\mathbf{n}} v=0$ on $\partial \Omega$.
We first derive an identity for $I_H$. Consider the vector field
\[
    X:=H'(v)|\nabla v|^2\nabla v.
\]
By the product rule,
\begin{align*}
    \operatorname{div}X
    &=
    H''(v)|\nabla v|^4
    +
    H'(v)\nabla|\nabla v|^2\cdot\nabla v
    +
    H'(v)|\nabla v|^2\Delta v.
\end{align*}
Since
\[
    \frac12\nabla|\nabla v|^2=D^2v\,\nabla v,
\]
we obtain
\[
    \operatorname{div}X
    =
    H''(v)|\nabla v|^4
    +
    2H'(v)  D^2v\,\nabla v \cdot \nabla v
    +
    H'(v)|\nabla v|^2\Delta v.
\]
Integrating over $\Omega$ and using the divergence theorem gives
\[
    \int_\Omega\operatorname{div}X\ \dx
    =
    \int_{\partial\Omega}
    H'(v)|\nabla v|^2\partial_{\mathbf{n}} v\ 
    \mathrm{d}\mathcal H^{d-1}
    =
    0,
\]
where the last equality follows from the homogeneous Neumann boundary
condition. Therefore,
\begin{equation}
\label{eq:A-identity}
    I_H
    =
    -\int_\Omega H''(v)|\nabla v|^4\ \dx
    -
    2\int_\Omega
    H'(v)
     D^2v\,\nabla v \cdot \nabla v \ \dx.
\end{equation}
Adding $\int_\Omega|\Delta v|^2\ \dx$ to both sides and using Lemma \ref{lemma H2}, we have
\begin{align}
    I_H+\int_\Omega|\Delta v|^2\ \dx
    \geq
    \int_\Omega
    \Big[
        |D^2v|_F^2
        -
        2H'(v)
         D^2v\,\nabla v \cdot \nabla v
        -H''(v)|\nabla v|^4
    \Big] \ \dx.
\label{eq:starting-cances}
\end{align}

We now use once again \eqref{eq:A-identity}, equivalently written as
\begin{equation}
\label{eq:zero-identity}
    0
    =
    \int_\Omega
    \Big[
        H''(v)|\nabla v|^4
        +
        2H'(v)
        D^2v\,\nabla v \cdot \nabla v
        +
        H'(v)|\nabla v|^2\Delta v
    \Big]\ \dx.
\end{equation}
Multiplying \eqref{eq:zero-identity} by $2/(d+2)$ and adding it to
the right-hand side of \eqref{eq:starting-cances}, we obtain
\begin{align}
    I_H+\int_\Omega|\Delta v|^2\ \dx
    \geq
    \int_\Omega
    \Big[
        |D^2v|_F^2
        &-
        \frac{2d}{d+2}
        H'(v)
         D^2v\,\nabla v \cdot \nabla v 
        \nonumber\\
        &+
        \frac{2}{d+2}
        H'(v)|\nabla v|^2\Delta v
        -
        \frac{d}{d+2}
        H''(v)|\nabla v|^4
    \Big] \ \dx.
\label{eq:cances-combination}
\end{align}

Define the matrix
\[
    \mathbb{D}(v)
    :=
    D^2v-\frac{\Delta v}{d}\mathbbm{I}_d,
\]
where $\mathbb{I}_d$ is the identity matrix.
Since the matrix $D^2 v$ is symmetric,
and 
$\Delta v= \textrm{tr}(D^2 v)$ leads to $\textrm{tr}\left(  \mathbb{D}(v) \right)=0$, we infer that
\begin{equation}
\label{eq:R-norm}
\left|D^2 v \right|_F^2
= \textrm{tr}\left[ \left(  \mathbb{D}(v)  +\dfrac{\Delta v}{d} \mathbbm{I}_d\right)^2\right]
= \left|  \mathbb{D}(v)  \right|_F^2+\dfrac{|\Delta v|^2}{d}
\end{equation}
and
\begin{equation}
\label{eq:R-mixed}
    D^2v\,\nabla v \cdot \nabla v
    =
  \mathbb{D}(v) \, \nabla v \cdot \nabla v
    +
    \frac{\Delta v}{d}|\nabla v|^2.
\end{equation}
Substituting \eqref{eq:R-norm} and \eqref{eq:R-mixed} into
\eqref{eq:cances-combination}, the mixed terms involving
$H'(v)|\nabla v|^2\Delta v$ cancel exactly. Hence, we arrive at
\begin{align}
    I_H+\int_\Omega|\Delta v|^2\ \dx
    & \geq \
    \frac1d \int_\Omega|\Delta v|^2\ \dx
    \nonumber\\
    &\quad +
    \int_\Omega
    \Big[
        | \mathbb{D}(v) |_F^2
        -
        \frac{2d}{d+2}
        H'(v)   \mathbb{D}(v) \nabla v \cdot \nabla v
        -
        \frac{d}{d+2}
        H''(v)|\nabla v|^4
    \Big]\ \dx.
\label{eq:trace-free}
\end{align}

Set
\[
    a:=\frac{d}{d+2}.
\]
The last integrand can be completed to a square as follows
\begin{align}
    &| \mathbb{D}(v) |_F^2
    -
    2a H'(v)  \mathbb{D}(v) \nabla v \cdot \nabla v
    -
    aH''(v)|\nabla v|^4
    \nonumber\\
    &\qquad=
    \left|
         \mathbb{D}(v) -aH'(v)\nabla v\otimes\nabla v
    \right|_F^2
    +
    a
    \Big[
        -H''(v)-a|H'(v)|^2
    \Big]
    |\nabla v|^4.
\label{eq:square-completion}
\end{align}
Since the first term on the right-hand side of
\eqref{eq:square-completion} is non-negative, we infer from
\eqref{eq:trace-free} that
\begin{align}
    I_H+\int_\Omega|\Delta v|^2\ \dx
    &\geq
    \frac1d \int_\Omega|\Delta v|^2\ \dx
    \nonumber\\
    &\quad 
    +
    \frac{d}{d+2}
    \int_\Omega
    \left[
        -H''(v)
        -
        \frac{d}{d+2}|H'(v)|^2
    \right]
    |\nabla v|^4\ \dx,
\end{align}
which proves \eqref{eq:refined-lemma}.
\end{proof}

We aim to apply Lemma \ref{refined-lemma} with a suitable function $H$ such that $H'(s)=- \dfrac{\sin (s)}{\cos(s)}$ on $\left(-\frac{\pi}{2}, \frac{\pi}{2}\right)$. 
For any $0<\varepsilon<1$ fixed, define
$$ 
J= I_\varepsilon :=\left[ -\arcsin \left(\frac{1}{\sqrt{1+\ve}} \right), \arcsin \left(\frac{1}{\sqrt{1+\ve}} \right) \right] \subset \left(-\frac{\pi}{2}, \frac{\pi}{2}\right)
$$
and 
$$
H: I_\varepsilon \to \mathbbm{R}, \qquad H(s)= \ln \left( \cos(s) \right).
$$
Since
$$
\cos: I_\varepsilon \to 
\left[ \sqrt{ \frac{\varepsilon}{1+\varepsilon}}, 1 \right],
$$
where we have used that $\cos (\arcsin (x))= \sqrt{1 - x^2}$, for any $x \in [-1,1]$, we deduce that $H \in W^{2,\infty}(I_\varepsilon)$.
Then, we compute 
\[
    -H''(s)-\frac{3}{5}(H')^2(s)
    =\frac{1}{\cos^2(s)}-\frac{3}{5}\frac{\sin^2(s)}{\cos^2(s)}
    =\frac{3}{5}+\frac{2}{5}\frac{1}{\cos^2(s)}, \quad \forall \, s \in I_\varepsilon.
    \]
    
Let $v(x)=\phi_\varepsilon(u_\varepsilon(x))\in I_\varepsilon$ for every $x\in\overline{\Omega}$.
Then, recalling that $\frac{1}{\cos^2(\phi_\ve(u_\ve))}=\dfrac{1+\varepsilon}{1+\varepsilon - u_\ve^2}$, an application of Lemma \ref{refined-lemma} yields
\begin{align*}
&\int_\Omega |\Delta \phi_\ve(u_\ve)|^2 \ \dx - \int_\Omega \dfrac{u_\ve}{\sqrt{1- u_\ve^2+\ve}}|\nabla \phi_\ve(u_\ve)|^2 \Delta \phi_\ve(u_\ve)\ \dx
\\
&
\quad 
\geq
    \frac13 \int_\Omega|\Delta \phi_\varepsilon(u_\varepsilon)|^2\ \dx
  +
    \frac35 \int_\Omega
   \Big[ \frac{3}{5}+\frac{2}{5}\frac{1}{\cos^2( \phi_\ve(u_\ve))}\Big]
    |\nabla  \phi_\ve(u_\ve)|^4\ \dx
 \\
 &  
\quad 
=
    \frac13 \int_\Omega|\Delta \phi_\varepsilon(u_\varepsilon)|^2\ \dx
  +
    \int_\Omega
   \Big[ \frac{9}{25}
+\frac{6}{25}
\frac{1+\varepsilon}
{1+\varepsilon-u_\varepsilon^2} \Big]
    |\nabla  \phi_\ve(u_\ve)|^4\ \dx.
\end{align*}

Next, going back to \eqref{SE 1}, we show that the second integral on the left-side is non-negative almost everywhere on $(0,T)$. Let us consider the Lipschitz truncation of $u_\ve$ defined by 
\begin{equation}
\label{trunc}
u_\ve^n := h_n(u_\ve), \quad \mbox{where} \quad h_n(s)=\begin{cases}
        -1+\dfrac{1}{n} \quad &\mbox{if }s< -1+\dfrac{1}{n},\\[3pt]
        s&\mbox{if } -1+\dfrac{1}{n}\le s \le 1-\dfrac{1}{n},\\[3pt]
        1-\dfrac{1}{n} \quad &\mbox{if }s> 1-\dfrac{1}{n}.
    \end{cases}
\end{equation}
It is well known that, for $n \to + \infty$,
\begin{align}
    &u_\ve^n \to u_\ve \quad \mbox{ a.e. in }\Omega_T, \label{conv un-eps}\\[3pt]
    &\nabla u_\ve^n = \nabla u_\ve \cdot \mathbf{1}_{\left[ -1+\frac{1}{n},1-\frac{1}{n}\right]}(u_\ve) \to \nabla u_\ve \quad \mbox{ a.e. in }\Omega_T.\label{conv nabla un-eps}
\end{align}
Let $\psi \in C^\infty_c(0,T)$ be a test function such that $\psi \geq 0$. Since $|f(u_\ve^n)| \leq |f(u_\ve)|$ almost everywhere in $\Omega_T$ and $f(u_\ve) \in L^2(\Omega_T)$, we have 
\begin{align*}
- \int_0^T \int_\Omega f(u_\ve) \Delta u_\ve \  \dx \ \psi(t) \ \dt
=- \lim_{n \to \infty} \int_0^T \int_\Omega f(u_\ve^n) \Delta u_\ve \ \psi(t) \ \dx  \dt.
\end{align*}
On the other hand, integrating by parts and using \eqref{conv nabla un-eps}, it follows that
\begin{align}
- \lim_{n \to \infty} \int_0^T  \int_\Omega f(u_\ve^n) \Delta u_\ve \ \psi(t) \ \dx\dt
&=\lim_{n \to \infty}  \int_0^T \int_\Omega f'(u_\ve^n) \nabla u_\ve^n \cdot \nabla u_\ve \ \psi(t) \ \dx \dt
\notag
\\[3pt]
& =\lim_{n \to \infty}  \int_0^T  \int_\Omega f'(u_\ve^n) |\nabla u_\ve^n|^2 \  \psi(t) \ \dx \dt \geq 0,
\label{f'-pos}
\end{align}
which implies the desired claim.

Collecting the above computations, we derive from \eqref{SE 1} that the following inequality
\begin{align*}
\frac23 \int_\Omega |\Delta \phi_\ve(u_\ve)|^2 \ \dx 
      &+\dfrac{18}{25}\int_\Omega  |\nabla \phi_\ve(u_\ve)|^4 \ \dx
 \\[3pt]
 &\leq 
 \int_\Omega \sqrt{b_\ve(u_\ve)} \left|\nabla \phi_\ve(u_\ve) \cdot \nabla \mu_\ve \right| \ \dx + \theta_0 \int_\Omega |\nabla u_\ve |^2 \ \dx
\end{align*}
holds almost everywhere on $(0,T)$. Now, computing the square and integrating over the time interval $(0,T)$, we deduce that
	\begin{align}
& \frac{4}{9} \int_0^T \left( \int_\Omega |\Delta \phi_\ve(u_\ve)|^2 \ \dx \right)^2  \dt
+\dfrac{324}{625} \int_0^T \left( \int_\Omega  |\nabla \phi_\ve(u_\ve)|^4 \ \dx \right)^2 \dt
\notag
	 \\[3pt]
&\quad \le  2  \int_0^T \left( \int_\Omega \sqrt{b_\ve(u_\ve)}
\left| \nabla \phi_\ve(u_\ve) \cdot \nabla \mu_\ve  \right| \ \dx \right)^2 \dt + 2  \theta_0^2 \int_0^T \left( \int_\Omega |\nabla u_\ve |^2 \ \dx \right)^2 \dt.
	\label{SE 7}
	\end{align}
Here we have used on the left-hand side that $a^2+b^2 \leq (a+b)^2$ for any $a,b \in \mathbbm{R}_+$.	
Exploiting the Cauchy-Schwarz inequality, we arrive at
	\begin{align}
		&\frac49 \int_0^T \left\| \Delta \phi_\ve(u_\ve) \right\|_{L^2(\Omega)}^4 \ \dt
		+\dfrac{4}{9} \int_0^T  \left\| \nabla \phi_\ve(u_\ve) \right\|_{L^4(\Omega)}^8 \ \dt
		\notag
		 \\[3pt]
		&\quad \le 2 \int_0^T \left\| \nabla \phi_\ve(u_\ve)\right\|_{L^2(\Omega)}^2 \Big\| \sqrt{b_\ve(u_\ve)} \nabla \mu_\ve \Big\|_{L^2(\Omega)}^2 \ \dt + 2 \theta_0^2 \int_0^T \| \nabla u_\ve \|_{L^2(\Omega)}^4 \ \dt.  
\label{SE 8}
	\end{align}
By using \eqref{E3}, \eqref{E2tris} and \eqref{E2cinque}, we derive that
\begin{equation}
    \label{delta phi(ueps)}
    \|\Delta \phi_\ve(u_\ve)\|_{L^4(0,T;L^2(\Omega))} + 
	\| \nabla \phi_\ve(u_\ve)\|_{L^8(0,T;L^4(\Omega))}
     \le C.
\end{equation}
Since  $|\phi_\ve(u_\ve)| \le |\phi(u_\ve)| \leq \frac{\pi}{2}$, and thanks to the homogeneous Neumann boundary conditions for $\phi_\ve(u_\ve)$, we conclude that
\begin{equation}
    \label{phiH2}
    \|\phi_\ve(u_\ve)\|_{L^4(0,T;H^2(\Omega))} \le C
\end{equation}
and
    \begin{equation}
        \label{phiW14}
        \|\phi_\ve(u_\ve)\|_{L^8(0,T;W^{1,4}(\Omega))} \le C.
    \end{equation}
    
We are now in the position to deduce similar bounds for  $u_\ve$.
By \eqref{delta phi}, we have
\begin{align}
    \Delta u_\ve &= \dfrac{1}{\phi_\ve'(u_\ve)}\Delta \phi_\ve(u_\ve) - \dfrac{\phi_\ve''(u_\ve)}{\phi_\ve'(u_\ve)}|\nabla u_\ve|^2
\notag    
    \\[3pt]
    &= \sqrt{1 -u^2_\ve+\ve} \, \Delta \phi_\ve(u_\ve)- \dfrac{\phi_\ve''(u_\ve)}{(\phi_\ve'(u_\ve))^3} |\nabla \phi_\ve(u_\ve)|^2
\notag    
     \\[3pt]
    &=\sqrt{1 -u^2_\ve+\ve} \,  \Delta \phi_\ve(u_\ve)- u_\ve |\nabla \phi_\ve(u_\ve)|^2.
    \label{relation Delta u}
\end{align}
Since $\sqrt{1 -u^2_\ve+\ve}$ and $u_\ve$ are uniformly bounded in $L^\infty(\Omega_T)$, exploiting the uniform estimates of $\Delta \phi_\ve(u_\ve)$ and $|\nabla \phi_\ve(u_\ve)|^2$ in $L^4(0,T; L^2(\Omega))$ (cf. \eqref{delta phi(ueps)}), we infer that
\begin{equation}
    \label{uH2}
    \|u_\ve\|_{L^4(0,T;H^2(\Omega))} \le C.
\end{equation}
Lastly, in light of \eqref{B1.2} and \eqref{uH2}, a direct application of the Gagliardo-Nirenberg interpolation inequality \eqref{GNI L4} yields
   \begin{equation}
        \label{uW14}
        \|u_\ve\|_{L^8(0,T;W^{1,4}(\Omega))} \le C.
    \end{equation}
   
 \medskip  

\begin{remark}
We observe that the above method reveals an additional uniform estimate (not mentioned in \cite{CM2023}) of the term
$$
\int_0^T\left( \int_\Omega \dfrac{1}{1- u_\ve^2+\varepsilon} \ |\nabla \phi_\ve(u_\ve)|^4 \ \dx \right)^2 \, \dt \leq C,
$$
which exhibits a more singular character as $\ve \to 0$ than simply $\| \nabla \phi_\ve(u_\ve) \|_{L^8(0,T; L^4(\Omega))}$.
\end{remark}

\subsubsection{General domains.} We perform an entropy estimate inspired by the work of Abels, Depner and Garcke \cite{ADG2013-2}.
For any $0<\varepsilon<1$, we define the function
$$
B_\varepsilon: [-1,1] \to \mathbbm{R}, \qquad B_\varepsilon(s)= \frac{1}{\sqrt{1+\varepsilon}} - \frac{\sqrt{1-s^2+\varepsilon}}{1+\varepsilon}.
$$
We observe that
\begin{equation}
\label{B_ve prop}
0\leq B_\varepsilon(s)\leq 1, \quad \forall \, s \in [-1,1]
\end{equation}
and
$$
B_\varepsilon'(s)= \frac{s}{(1+\varepsilon)\sqrt{1-s^2+\varepsilon}},
\quad
 B_\varepsilon''(s)= \frac{1}{\left( 1-s^2+\varepsilon\right)^\frac32}=\frac{1}{\left( b_\varepsilon(s)\right)^\frac32}, \quad \forall \, s \in [-1,1].
$$
Since $B_\varepsilon$ is a convex function, by using \cite[Lemma 4.1]{RS}, we have
$$
\int_\Omega B_\ve (u_\ve(t)) \ \dx 
=\int_\Omega B_\ve (u_0) \ \dx +  \int_0^t \langle \partial_t u_\ve, B'_\ve(u_\ve) \rangle \ \mathrm{d}s, \quad \forall \, t \in [0,T].
$$
By exploiting \eqref{ueps} and noticing that $B'_\varepsilon(u_\ve) \in L^2(0,T;H^1(\Omega))$, it follows that 
\begin{align*}
\int_0^t \langle \partial_t u_\ve, B'_\ve(u_\ve) \rangle \ \mathrm{d}s
&= 
-  \int_0^t \int_\Omega b_\ve (u_\ve) \nabla \mu_\ve \cdot  \nabla B'_\ve(u_\ve)  \ \dx \mathrm{d}s
\\[3pt]
&= 
-  \int_0^t \int_\Omega b_\ve (u_\ve) B''_\ve(u_\ve) \nabla \mu_\ve \cdot \nabla u_\ve \ \dx \mathrm{d}s
\\[3pt]
&= 
-  \int_0^t \int_\Omega \phi'_\ve (u_\ve) \nabla \mu_\ve \cdot \nabla u_\ve  \ \dx \mathrm{d}s
\\[3pt]
&= 
-  \int_0^t \int_\Omega  \nabla \mu_\ve \cdot \nabla \phi_\ve(u_\ve) \ \dx \mathrm{d}s,
\quad \forall \, t \in [0,T].
\end{align*}

Now, multiplying \eqref{newformulation} by $-\Delta \phi_\ve(u_\ve)$ and integrating over $\Omega_t$, we obtain the integral identity
\begin{align} 
		2 \int_0^t \int_\Omega &\phi_\ve'(u_\ve) |\Delta \phi_\ve(u_\ve)|^2 \ \dx\ds - \int_0^t \int_\Omega f(u_\ve) \Delta \phi_\ve(u_\ve) \  \dx \ds \notag \\[3pt]
		&= - \int_0^t \int_\Omega  \mu_\ve  \Delta \phi_\ve(u_\ve) \ \dx \ds
		- \theta_0 \int_0^t \int_\Omega  u_\ve \Delta \phi_\ve(u_\ve) \ \dx\ds.
\label{SE 1-new}
\end{align}
Reasoning as in \eqref{trunc}-\eqref{f'-pos}, for any test function $\psi\in C^\infty_c(0,t)$ such that $\psi \geq 0$,
we have
\begin{align*}
- \int_0^t \int_\Omega f(u_\ve) \Delta \phi_\ve(u_\ve) \  \dx \ \psi(t)  \ds
\notag
&=
- \lim_{n \to \infty} \int_0^t  \int_\Omega f(u_\ve^n) \Delta \phi_\ve(u_\ve) \ \psi(t) \ \dx \ds
\notag 
\\[3pt]
&=\lim_{n \to \infty}  \int_0^t \int_\Omega f'(u_\ve^n) \phi'_\ve(u_\ve) \nabla u_\ve^n \cdot \nabla u_\ve \ \psi(t) \ \dx\ds
\notag
\\[3pt]
& =\lim_{n \to \infty}  \int_0^t  \int_\Omega  f'(u_\ve^n) \phi'_\ve(u_\ve) |\nabla u_\ve^n|^2 \ \psi(t) \ \dx \ds \geq 0.
\end{align*}
Besides, we observe that
\begin{align*}
- \theta_0 \int_0^t \int_\Omega u_\ve \Delta \phi_\ve(u_\ve) \ \dx\ds
= \theta_0 \int_0^t \int_\Omega  \nabla u_\ve \cdot \nabla \phi_\ve(u_\ve) \ \dx\ds.
\end{align*}
Thus, by combining the above relations, we arrive at
\begin{align}
&\int_\Omega B_\ve (u_\ve(t)) \ \dx 
+ 2 \int_0^t \int_\Omega \phi_\ve'(u_\ve) |\Delta \phi_\ve(u_\ve)|^2 \ \dx\ds
\notag
\\[3pt]
&\quad \leq \int_\Omega B_\ve (u_0) \ \dx +
  \theta_0 \int_0^t \int_\Omega  \nabla u_\ve \cdot \nabla \phi_\ve(u_\ve) \ \dx\ds, \quad \forall \, t \in [0,T].
\end{align}
In light of \eqref{I1} and \eqref{B_ve prop}, and by exploiting \eqref{E2bis} and \eqref{E2quattro}, we conclude that
\begin{equation}
\label{NEW-deltaphi}
\int_0^T \int_\Omega \phi_\ve'(u_\ve) |\Delta \phi_\ve(u_\ve)|^2 \ \dx\dt \leq C.
\end{equation}  
Since $\phi'_\ve(u_\ve)\geq \frac{1}{\sqrt{2}}$, we deduce in particular that
\begin{equation}
    \label{NEW-phiH2}
    \|\phi_\ve(u_\ve)\|_{L^2(0,T;H^2(\Omega))} \le C
\end{equation}
and, by interpolation (cf.\eqref{GNI L4}),
\begin{equation}
    \label{NEW-phiW14}
    \| \nabla \phi_\ve(u_\ve)\|_{L^4(0,T;L^4(\Omega))} \le C.
\end{equation}
Finally, by using \eqref{relation Delta u}, it is easily seen that
\begin{equation}
    \label{NEW-uH2}
    \|u_\ve\|_{L^2(0,T;H^2(\Omega))} \le C
\end{equation}
and 
\begin{equation}
    \label{NEW-uW14}
    \| \nabla u_\ve\|_{L^4(0,T;L^4(\Omega))} \le C.
\end{equation}

\subsection{Estimate of the free energy potential}
\label{ss3}
The procedure that yields a bound for $f(u_\ve)$ is fairly standard in the case of a strictly positive mobility. This is not expected in the case of degenerate mobility. Nonetheless, we obtain an estimate  of $f(u_\ve)$ with an explicit control of the dependence on $\ve$, which will allow us to pass to the limit in a suitable weak formulation. 

Following the classical approach (see \cite{MZ2004}), there exists a positive constant $C_f$, such that
\begin{equation}
    \label{f 1}
    \int_\Omega |f(u_\ve)| \ \dx \le C_f \left| \int_\Omega f(u_\ve) \left(u_\ve-\overline{u_0} \right)\ \dx \right|+C_f,
\end{equation}
where $C_f \to + \infty$ as $|\overline{u_0}|\to 1$.
Recalling from \eqref{newformulation} that
$$f(u_\ve)= \mu_\ve + 2 \phi'_\ve(u_\ve)\Delta \phi_\ve(u_\ve) + \theta_0u_\ve\quad \mbox{ a.e. in }\Omega_T,$$
we then have
\begin{align*}
    &\left| \int_\Omega f(u_\ve) \left(u_\ve-\overline{u_0} \right) \ \dx \right|
    \\[3pt]
    &\quad 
    \le \left| \int_\Omega \mu_\ve\left(u_\ve-\overline{u_0} \right) \ \dx \right|\ 
    + \left| 2\int_\Omega \phi'_\ve(u_\ve)\Delta \phi_\ve(u_\ve) \left(u_\ve-\overline{u_0} \right) \ \dx \right|
    + \theta_0\left|  \int_\Omega u_\ve \left(u_\ve-\overline{u_0} \right) \ \dx  \right
    |\\[3pt]
    &\quad 
    = \left| \int_\Omega u_\ve \left( \mu_\ve-\overline{\mu_\ve} \right) \ \dx 
    \right| + \left| 2\int_\Omega \dfrac{1}{\sqrt{1-u^2_\ve+\ve}}\Delta \phi_\ve(u_\ve) \left(u_\ve-\overline{u_0} \right) \ \dx \right|
    +\theta_0 \left| \int_\Omega u_\ve \left(u_\ve-\overline{u_0} \right) \ \dx  \right|
    \\[3pt]
    &\quad \le C \|\nabla \mu_\ve \|_{L^2(\Omega)}
    +2 \left\| \dfrac{1}{\sqrt{1-u^2_\ve+\ve}} \right\|_{L^\infty(\Omega)} 
    \|\Delta \phi_\ve(u_\ve)\|_{L^2(\Omega)} 
    \|u_\ve - \overline{u_0}\|_{L^2(\Omega)}
    + \theta_0 \|u_\ve - \overline{u_0}\|^2_{L^2(\Omega)}.
\end{align*}
Recalling that $b_\ve(s) \ge \ve$ and $\sqrt{1-u^2_\ve+\ve} \geq \sqrt{\ve}$, we infer that
\begin{align*}
    \left| \int_\Omega f(u_\ve)(u_\ve-\overline{u_0}) \ \dx \right|
    &\le C\left(\dfrac{1}{\sqrt{\ve}} \left\|\sqrt{b_\ve(u_\ve)}\nabla \mu_\ve \right\|_{L^2(\Omega)}+ \dfrac{1}{\sqrt{\ve}}\|\Delta \phi_\ve(u_\ve)\|_{L^2(\Omega)}+1\right)\!.
\end{align*}
Finally, in light of \eqref{f 1}, and owing to  \eqref{E3} and \eqref{delta phi(ueps)}, an integration in time gives
$$ 
\|f(u_\ve)\|_{L^2(0,T; L^1(\Omega))}\le C\left(\dfrac{1}{\sqrt{\ve}}+1\right)\!,
$$
with $C$ independent of $\ve$. Thus, since $\ve<1$, we conclude that
\begin{equation}
    \label{f L1}
     \|f(u_\ve)\|_{L^2(0,T; L^1(\Omega))}\le C\dfrac{1}{\sqrt{\ve}}.
\end{equation}

\subsection{Convergences} 
We now deduce convergence properties for $\{ u_{\ve} \}_{\ve>0}$ and $\{ \phi_\ve (u_{\ve}) \}_{\ve>0}$, which are needed in the final passage to the limit.
    
	\begin{itemize}
	    \item Convergences for $u_\ve$. First of all, passing to suitable subsequences not relabeled, as a direct consequence of \eqref{E2tris} and \eqref{NEW-uH2}, we infer that 
        \begin{align}
            u_{\ve} \overset{\star}{\rightharpoonup} u &\quad \mbox{in } L^{\infty}(0,T; H^{1}(\Omega)),  \label{conv2}\\[3pt]
			u_{\ve} \rightharpoonup u &\quad \mbox{in } L^{2}(0,T; H^{2}(\Omega)). \label{conv3}
		\end{align}     
    From \eqref{E3}, by direct comparison, we deduce that
    \begin{equation}
        \label{partialt}
        \|\partial_t u_\ve\|_{L^2(0,T;(H^{1}(\Omega))')} \le C,
    \end{equation}
    which, in turn, yields
    \begin{equation}\label{conv4}
        \partial_t u_{\ve} \rightharpoonup \partial_tu \quad \mbox{in } L^{2}(0,T; (H^{1}(\Omega))'). 
    \end{equation}
    By \eqref{NEW-uH2} and \eqref{partialt}, an application of the Aubin-Lions Lemma gives 
    \begin{equation}
        \label{conv1}
        u_{\ve} \to u \quad \mbox{in } L^2(0;T; W^{1,p}(\Omega)), \quad \forall \, p<6;
    \end{equation}
    while, from \eqref{E2tris} and \eqref{partialt},
    $$u_\ve \to u \quad \textrm{ in }C([0,T];L^2(\Omega)),$$
    and in particular, 
    \begin{equation}
        \label{uae}
        u_\ve \to u \quad \textrm{ a.e. in }\Omega_T.
    \end{equation}
    As a consequence, since $|u_\ve(x,t)|<1$ a.e. in $\Omega_T$, we also infer that 
    \begin{equation}
        \label{u tra -1 e 1}
        -1\le u\le 1 \quad \textrm{ a.e. in }\Omega_T.
    \end{equation}
    We finally highlight that, since $u_\ve$ is uniformly bounded in $\Omega_T$, Lemma \ref{convergenza alla Vitali} entails that
    \begin{equation}
        \label{mega}
        u_\ve \to u \quad \mbox{ in }L^\alpha(0,T;L^\beta(\Omega)), \quad \forall \,  1 \le \alpha,\beta < \infty.
    \end{equation}    
    
    \item Convergence of $b_\ve(u_\ve)$. Since $b_\ve(s)= 1-s^2+\ve$ is continuous and 
    $$b_\ve(s) \to b(s) \quad\textrm{ uniformly in }[-1,1],$$
    by the pointwise convergence \eqref{uae}, it is immediate that
    \begin{equation}
    \label{pointwise conv b}
    b_\ve(u_\ve)\to b(u) \quad \textrm{ a.e. in }\Omega_T.
    \end{equation}
    Furthermore, since $b_\ve$ is uniformly bounded, we conclude that
    \begin{equation}
        \label{convb}
        b_\ve(u_\ve) \to b(u) \quad  \textrm{ in }L^p(\Omega_T), \quad \forall \, p\in [2,\infty).
    \end{equation}
    \item Convergences for $\phi_\ve(u_\ve)$. By definition, $\phi_\ve(s)= \arcsin \left( \dfrac{s}{\sqrt{1+\ve}}\right)$ for $s\in [-1,1]$, thus
    $$\phi_\ve(s) \to \phi(s)=\arcsin(s) \quad\textrm{ uniformly in }[-1,1].$$
    Exploiting once again the almost everywhere convergence \eqref{uae}, we also have
    $$\phi_\ve(u_\ve) \to \phi(u)\quad\textrm{ a.e. in }\Omega_T.$$
    Moreover, since $\phi_\ve(u_\ve)$ is uniformly bounded, it follows that
    \begin{equation}
        \label{convphi}
        \phi_\ve(u_\ve) \to \phi(u) \quad \textrm{ in }L^p(\Omega_T), \quad 
        \forall \, p\in [2,\infty).
    \end{equation}
    In light of the uniform bound \eqref{NEW-phiH2}, we can also deduce that
    \begin{equation}
        \phi_{\ve}(u_{\ve}) \rightharpoonup \phi(u) \quad \mbox{in } L^2(0,T; H^2(\Omega)).\label{conv6}
    \end{equation}
   Furthermore, from estimate \eqref{NEW-phiW14}, passing to a suitable subsequence, we obtain
    $$
    \nabla \phi_\ve(u_\ve) \rightharpoonup \nabla \phi(u) \quad \textrm{ in }L^4(0,T; L^4(\Omega)).
    $$
    The next goal is to achieve strong convergences properties. 
     To this end, by the Gagliardo-Nirenberg inequality \eqref{GNI L3}
 $$
 \|\nabla \phi_\ve(u_\ve)-\nabla \phi(u)\|_{L^3(\Omega)}\le C \|\phi_\ve(u_\ve)-\phi(u)\|^{\frac12}_{H^2(\Omega)}\|\phi_\ve(u_\ve)-\phi(u)\|^{\frac12}_{L^6(\Omega)}.
 $$
Then, an application of the H\"older inequality yields
 \begin{align*}
   \int_0^T & \|\nabla \phi_\ve(u_\ve)-\nabla \phi(u)\|_{L^3(\Omega)}^3 \ \dt 
   \\ 
        & \qquad
        \leq 
        C\left( \int_0^T\|\phi_\ve(u_\ve)-\phi(u)\|^{2}_{H^2(\Omega)}\ \dt\right)^{\frac34} 
        \left( \int_0^T\|\phi_\ve(u_\ve)-\phi(u)\|^{6}_{L^6(\Omega)}\ \dt\right)^{\frac14}\!.
    \end{align*}
  Owing to the strong convergence \eqref{convphi} and the uniform control \eqref{NEW-phiH2} of $\phi_\ve(u_\ve)$, we conclude that
    \begin{equation}
        \label{conv5}
        \nabla \phi_{\ve}(u_{\ve}) \to \nabla\phi(u) \quad \mbox{in } L^3(\Omega_T),
    \end{equation}
    and, in particular, 
    \begin{equation}
        \label{pointwise convergence of nablaphi}
        \nabla \phi_\ve(u_\ve) \to \nabla \phi(u) \quad \textrm{ a.e. in }\Omega_T.
    \end{equation}
    Thus, by \eqref{NEW-phiW14} and the pointwise convergence \eqref{pointwise convergence of nablaphi}, in view of Lemma \ref{convergenza alla Vitali}, we also get
    \begin{equation}
        \label{morefortegradphi}
        \nabla \phi_\ve(u_\ve) \to \nabla \phi(u) \quad \textrm{ in }L^\alpha(0,T;L^\beta(\Omega)), \quad \forall \ 1 \le \alpha<4, \ 1 \le \beta <4.
    \end{equation}
    
    \item Identification of $\nabla \phi(u)$. We recall that
    $$
    \nabla \phi_\ve(u_\ve)= \dfrac{1}{\sqrt{1-u^2_\ve +\ve}} \nabla u_\ve \quad \textrm{ a.e. in }\Omega_T.
    $$
  Upon passing to a suitable subsequence, if necessary, we infer from \eqref{conv1} that
  \begin{equation}
        \label{pointwise convergence of nablau}
        \nabla u_\ve \to \nabla u \quad \textrm{ a.e. in }\Omega_T.
    \end{equation}
    In addition, we also have from \eqref{pointwise conv b} that
    \begin{equation}
    \dfrac{1}{\sqrt{1-u^2_\ve +\ve}} \to 
      \dfrac{1}{\sqrt{1-u^2}}
    \quad \textrm{ a.e. in } \{ |u|< 1 \} \subseteq \Omega_T,
    \end{equation}
    where $\{ |u|< 1 \}:= \{ (x,t) \in \Omega \times (0,T) :  |u(x,t)| < 1  \}.$
    Then, as a direct consequence, we get
    $$
     \nabla \phi_\ve(u_\ve) \to  \dfrac{1}{\sqrt{1-u^2}} \nabla u \quad \textrm{ a.e. in } \{ |u|< 1 \} \subseteq \Omega_T.
    $$
    On the other hand, we notice that
    $$
    \{  |u| = 1 \} = \{ u= 1 \} \cup \{ u = -1 \} = \left\{ \phi(u) = \dfrac{\pi}{2} \right\} \cup  \left\{ \phi(u) = -\dfrac{\pi}{2} \right\}.
    $$
 Since $\nabla \phi(u)=0$ almost everywhere in both sets $\left\{ \phi(u) = \dfrac{\pi}{2} \right\}$ and $\left\{ \phi(u) = - \dfrac{\pi}{2} \right\} $ (see, for instance, \cite{EPDE1998}), we derive that
    $$
    \nabla \phi(u)=0 \quad \text{a.e. in } \{  |u| = 1 \} .
    $$
    Therefore, we establish that 
    \begin{equation}
    \label{nabla phiu- ae}
    \nabla \phi(u)= 
    \begin{cases}
     \dfrac{1}{\sqrt{1-u^2}} \nabla u \quad &\text{ in } \{ |u|< 1 \}
     \\
     0 \quad &\text{ in } \{ |u|= 1 \}
    \end{cases}
    \quad \text{a.e. in  } \Omega_T.
    \end{equation}
	\end{itemize}
	
	\begin{remark}
	When $\Omega$ is convex, we can immediately deduce from \eqref{phiH2} and \eqref{uH2} that 
	$$          
 u\in L^4(0,T;H^2(\Omega))\quad \text{and} \quad \phi(u) \in L^4(0,T;H^2(\Omega)).
            $$  
	\end{remark}
    
   \subsection{Passage to the limit}
Thanks to the uniform bound \eqref{E3}, it is easily seen that
    $$\int_{\Omega_T} |b_\ve(u_\ve)\nabla \mu_\ve|^2\ \dx \dt \le 2 \int_{\Omega_T} b_\ve(u_\ve)|\nabla \mu_\ve|^2\  \dx \dt \le C, $$
    uniformly in $\ve$. 
    Therefore, there exists a subsequence, not relabeled, such that
 \begin{equation}
 \label{J-limit}
    b_\ve(u_\ve) \nabla \mu_\ve \rightharpoonup \boldsymbol{J} \quad \textrm{ in }L^2(\Omega_T).
    \end{equation}
 In light of the convergences gained before, we are now ready to let $\ve \to 0$ in \eqref{ueps}. More precisely, for any $\psi \in C_c^\infty(0,T)$ and $v \in H^1(\Omega)$, we know from \eqref{ueps} that
$$
\int_0^T \langle \partial_t u_{\ve},v \rangle \ \psi(t) \ \dt = - \int_0^T 
(b_{\ve}(u_{\ve})\nabla \mu_{\ve}, \nabla v) \ \psi(t) \ \dt.
 $$
 By exploiting \eqref{conv4} and \eqref{J-limit},
 we obtain
 $$
\int_0^T \langle \partial_t u ,v \rangle \ \psi(t) \ \dt = - \int_{\Omega_T}
 \boldsymbol{J} \cdot \nabla v \, \psi(t) \ \dx \dt.
 $$
 By a classical density argument, the weak formulation \eqref{weak formulation} readily follows.
 
Next, we identify the limit flux $\boldsymbol{J}$ in a distributional sense. As a first step, we introduce a suitable weak formulation. 
For $\boldsymbol{\eta} \in C_c^\infty\left( \overline{\Omega}\times (0,T); \mathbbm{R}^3\right)$,
with $\boldsymbol{\eta}\cdot \mathbf{n}=0$ on $\partial \Omega\times (0,T)$,
 \begin{align}
    \int_{\Omega_T} b_{\ve}(u_{\ve})\nabla \mu_{\ve} \cdot \boldsymbol{\eta} \  \dx \dt 
    &= \int_{\Omega_T} \nabla \Big( b_{\ve}(u_{\ve})\mu_\ve\Big)\cdot \boldsymbol{\eta} \ \dx \dt 
    - \int_{\Omega_T} b'_\ve (u_\ve)\mu_\ve\nabla u_\ve \cdot \boldsymbol{\eta}\ \dx \dt \notag 
    \\[3pt]
        &= - \int_{\Omega_T}  b_{\ve}(u_{\ve})\mu_\ve \div \,  \boldsymbol{\eta}\ \dx \dt
        - \int_{\Omega_T} b'_\ve (u_\ve)\mu_\ve\nabla u_\ve \cdot \boldsymbol{\eta}\ \dx \dt. \label{id J}
    \end{align}
Owing to \eqref{newformulation}, the first integral on the right-hand side of \eqref{id J} can be rewritten as
  \begin{align}
        -&\int_{\Omega_T} b_{\ve}(u_{\ve})\mu _\ve  \div \,  \boldsymbol{\eta}\ \dx \dt  \notag
        \\[3pt]
    &= 2\int_{\Omega_T} \sqrt{b_\ve(u_\ve)}\Delta \phi_\ve(u_\ve)\div \,   \boldsymbol{\eta} \ \dx \dt  - \int_{\Omega_T} b_\ve(u_\ve) f(u_\ve)\div \,   \boldsymbol{\eta}  \ \dx \dt   \notag  \\[3pt]
    &\quad  + \theta_0 \int_{\Omega_T} b_\ve(u_\ve) u_\ve \div \,  \boldsymbol{\eta} \ \dx \dt .
    \label{primo}
    \end{align}
 
Now, let us define
$$
G: [-1,1] \to \mathbbm{R}, \qquad G(s):= \int_0^{s} b_\ve'(\tau)f(\tau) \ \mathrm{d}\tau= \int_0^{s} b'(\tau)f(\tau) \ \mathrm{d}\tau, \quad \forall \, s \in [-1,1].
$$
Since $b'f \in L^1(-1,1)$, we notice that $G\in AC([-1,1])$, and then $G(u_\varepsilon)$ is well defined by \eqref{B1.2}. In addition, it follows from \eqref{B1.1}--\eqref{B1.3} that $\nabla G(u_\ve)= b'(u_\ve) f(u_\ve) \nabla u_\ve$ almost everywhere in $\Omega_T$. Then,
as for the second integral on the right-hand side of \eqref{id J}, using once again \eqref{newformulation} and an integration by parts, we obtain
\begin{align}
   - &\int_{\Omega_T} b'_\ve (u_\ve)\mu_\ve\nabla u_\ve \cdot \boldsymbol{\eta}\ \dx \dt \notag\\[3pt]
    &=  2 \int_{\Omega_T} b'_\ve(u_\ve) \Delta \phi_\ve(u_\ve)\nabla \phi_\ve(u_\ve)\cdot \boldsymbol{\eta} \ \dx \dt -\int_{\Omega_T} \nabla G(u_\ve) \cdot \boldsymbol{\eta} \ \dx \dt \notag
    \\[3pt]
    &\quad + \theta_0 \int_{\Omega_T} b'_\ve(u_\ve) u_\ve \nabla u_\ve \cdot \boldsymbol{\eta} \ \dx \dt \notag \\[3pt]
    &=  2 \int_{\Omega_T} b'_\ve(u_\ve) \Delta \phi_\ve(u_\ve)\nabla \phi_\ve(u_\ve)\cdot \boldsymbol{\eta}\ \dx \dt +\int_{\Omega_T} G(u_\ve)\div \,  \boldsymbol{\eta} \ \dx \dt  \notag
\\[3pt]
    &\quad     
    + \theta_0 \int_{\Omega_T} b'_\ve(u_\ve) u_\ve \nabla u_\ve\cdot \boldsymbol{\eta} \ \dx \dt. \label{secondo}
\end{align}
Furthermore, in light of \eqref{key relation}, an integration by parts reveals that, for any $s \in (-1,1)$,
\begin{equation}
\label{CHI E G}
G(s)= \int_0^s b'(\tau)f(\tau) \, \mathrm{d}\tau
= b(s)f(s)- \int_0^s b(\tau) f'(\tau) \, \mathrm{d}\tau
=b(s)f(s)- \theta s,
\end{equation}
which can be extended by continuity in $s=\pm 1$. Hence, we observe that
\begin{align}
&- \int_{\Omega_T} b_\ve(u_\ve) f(u_\ve)\div \,   \boldsymbol{\eta}  \ \dx \dt 
+ 
\int_{\Omega_T} G(u_\ve)\div \,  \boldsymbol{\eta} \ \dx \dt
\notag
\\[3pt]
&
\quad = - \int_{\Omega_T} \varepsilon f(u_\ve)\div \,   \boldsymbol{\eta}  \ \dx \dt  - \int_{\Omega_T} \theta u_\ve \div \,  \boldsymbol{\eta}  \ \dx \dt.
\label{Cancellazione}
\end{align}
On the other hand, we also notice that
\begin{align}
\theta_0 \int_{\Omega_T} b_\ve(u_\ve) u_\ve \div \,  \boldsymbol{\eta} \ \dx \dt
+
\theta_0 \int_{\Omega_T} b'_\ve(u_\ve) u_\ve \nabla u_\ve\cdot \boldsymbol{\eta} \ \dx \dt 
= 
- \theta_0 \int_{\Omega_T} b_\ve(u_\ve) \nabla u_\ve \cdot  \boldsymbol{\eta} \ \dx \dt.
\label{Cancellazione2}
\end{align}

In summary, from \eqref{id J}--\eqref{secondo} and \eqref{Cancellazione}--\eqref{Cancellazione2}, we end up with
    \begin{equation}
        \label{limitequi}
        \int_{\Omega_T} b_{\ve}(u_{\ve})\nabla \mu_{\ve} \cdot \boldsymbol{\eta} \  \dx \dt= \sum_{i=1}^{5} I_{i, \ve},
    \end{equation}
    where
        \begin{align*}
       I_{1, \ve}&=2\int_{\Omega_T} \sqrt{b_\ve(u_\ve)}\Delta \phi_\ve(u_\ve) \div \, \boldsymbol{\eta}\ \dx \dt,&
       I_{2, \ve}&= 2 \int_{\Omega_T} b'_\ve(u_\ve) \Delta \phi_\ve(u_\ve)\nabla \phi_\ve(u_\ve)\cdot \boldsymbol{\eta} \ \dx \dt ,
       \\[3pt]
       I_{3, \ve}& = -\varepsilon  \int_{\Omega_T}  f(u_\ve)\div \,   \boldsymbol{\eta}  \ \dx \dt , &       
       I_{4,\ve}&= - \theta \int_{\Omega_T} u_\ve \div \,  \boldsymbol{\eta}  \ \dx \dt,
        \\[3pt]
        I_{5, \ve}&=- \theta_0 \int_{\Omega_T} b_\ve(u_\ve) \nabla u_\ve \cdot  \boldsymbol{\eta} \ \dx \dt.
    \end{align*}

We proceed by studying the convergence of each integral. 
\begin{enumerate}[i)]
\item From the strong convergence \eqref{convb} of $\sqrt{b_\ve(u_\ve)}$ in $ L^p(\Omega_T)$, for all $p \in [2,\infty)$, and the weak convergence of $\Delta \phi_\ve(u_\ve)$ in $L^2(  \Omega_T)$,  it follows that
$$
I_{1,\ve} \to 2\int_{\Omega_T}  \sqrt{b(u)}\Delta \phi(u) \div \,  \boldsymbol{\eta} \ \dx \dt \quad \mbox{as }\ve \to 0.$$

 \item From the definition of $b_\ve$ in \eqref{a-b-def},
$$\int_{\Omega_T} b'_\ve(u_\ve) \Delta \phi_\ve(u_\ve) \nabla\phi_\ve(u_\ve) \cdot \boldsymbol{\eta}\ \dx \dt = -\int_{\Omega_T} 2u_\ve \Delta \phi_\ve(u_\ve) \nabla\phi_\ve(u_\ve) \cdot \boldsymbol{\eta} \ \dx \dt. $$
Furthermore, we can write
\begin{align*}
    \int_{\Omega_T}  \Big[u_\ve& \Delta \phi_\ve(u_\ve) \nabla \phi_\ve(u_\ve) - u \Delta \phi(u)\nabla \phi(u)\Big]\cdot \boldsymbol{\eta} \ \dx \dt \\[3pt]
    &=\int_{\Omega_T}  \Delta \phi_\ve(u_\ve) \Big[ u_\ve\nabla \phi_\ve(u_\ve)-u\nabla\phi(u)\Big]\cdot \boldsymbol{\eta} \ \dx \dt \\[3pt]
    &\hspace{0,5 cm}+ \int_{\Omega_T} u \nabla \phi(u) \Big[ \Delta \phi_\ve(u_\ve)-\Delta\phi(u)\Big]\cdot \boldsymbol{\eta} \ \dx \dt.
\end{align*}
Firstly, we observe that $\Delta \phi_\ve(u_\ve)$ is uniformly bounded in $L^2( \Omega_T )$. Moreover, in view of  \eqref{mega} with $\alpha=\beta=6$ and \eqref{morefortegradphi} with $\alpha=\beta=3$, we deduce the strong convergence
$$u_\ve \nabla\phi_\ve(u_\ve) \to u \nabla \phi(u) \quad \mbox{ in }L^2(  \Omega_T  ).$$
Hence, the first integral converges. Moreover,
$$\|u \nabla \phi(u)\cdot \boldsymbol{\eta}\|_{L^2(\Omega_T)}\le \|u\|_{L^\infty(\Omega_T)}\|\boldsymbol{\eta}\|_{L^\infty(\Omega_T)}\|\nabla \phi(u)\|_{L^2(\Omega_T)}\le C.$$
Therefore, owing to the weak convergence \eqref{conv6} of $\Delta \phi_\ve(u_\ve)$, the second integral converges as well, and we conclude that
$$I_{2, \ve} \to \int_{\Omega_T}  b'(u) \Delta \phi(u) \nabla \phi(u) \cdot \boldsymbol{\eta}\ \dx \dt.$$

\item We exploit the crucial estimate \eqref{f L1}. Although $\|f(u_\ve)\|_{L^1(\Omega_T)}$ may diverge as $\ve \to 0$, its growth, which is at most of order $\ve^{-\frac12}$, is more than compensated by the prefactor $\ve$. 
More precisely, 
\begin{align*}
    \left|   \ve  \int_{\Omega_T} f(u_\ve) \div \, \boldsymbol{\eta} \ \dx \dt \right|&\le \ve \|f(u_\ve)\|_{L^1(\Omega_T)} \|\div \, \boldsymbol{\eta}\|_{L^\infty(\Omega_T)} \\[3pt]
    &\le C\sqrt{\ve} \|\div \, \boldsymbol{\eta}\|_{L^\infty(\Omega_T)} \to 0, \quad \textrm{for }\ve \to 0.
\end{align*}
Combining the above estimates, we can conclude that
$$
I_{3, \ve} \to 0.
$$

\item In light of \eqref{conv3}, it is easily seen that
$$
 I_{4,\ve} \to  - \theta \int_{\Omega_T} u \div \,  \boldsymbol{\eta}  \ \dx \dt.
$$

\item 
We simply notice that 
\begin{align*}
&\|  b_\ve(u_\ve) \nabla u_\ve - b(u) \nabla u\|_{L^2(\Omega_T)}
\\[3pt]
&\quad \leq 
\|  b_\ve(u_\ve) - b(u) \|_{L^4(\Omega_T)}\| \nabla u_\ve\|_{L^4(\Omega_T)}
+
\|  b(u)\|_{L^\infty(\Omega_T)} \| \nabla u_\ve - \nabla u\|_{L^2(\Omega_T)}.
\end{align*}
Recalling that $b(u)$ is uniformly bounded in $L^\infty(\Omega_T)$, thanks to \eqref{NEW-uW14}, \eqref{conv1} and \eqref{convb}, we know that $b_\ve(u_\ve) \nabla u_\ve \to b(u) \nabla u$ strongly in $L^2(\Omega_T)$, which, in turn, entails 
$$
I_{5, \ve} \to - \theta_0 \int_{\Omega_T} b(u) \nabla u \cdot  \boldsymbol{\eta} \ \dx \dt.
$$
\end{enumerate}
Collecting the above results, the limit $u$ satisfies 
\begin{align*}
\int_{\Omega_T} \boldsymbol{J} \cdot \boldsymbol{\eta} \ \dx \dt &= \int_{\Omega_T} \left[2\sqrt{b(u)}\Delta \phi(u)-
\theta u \right]\div\,  \boldsymbol{\eta} \ \dx \dt  \notag 
\\[3pt]
& \quad + \int_{\Omega_T} 2 b'(u) \Delta \phi(u)\nabla \phi(u) \cdot \boldsymbol{\eta} 
- \theta_0 b(u) \nabla u \cdot \boldsymbol{\eta} \ \dx \dt, 
\end{align*}
for all $\boldsymbol{\eta} \in C_c^\infty\left( \overline{\Omega}\times (0,T); \mathbbm{R}^3\right)$ with $\boldsymbol{\eta} \cdot \mathbf{n}=0$ on $\partial \Omega \times(0,T)$. 
Thus, a final integration by parts implies that
\begin{align*}
\int_{\Omega_T} \boldsymbol{J} \cdot \boldsymbol{\eta} \ \dx \dt &= \int_{\Omega_T} \left[2\sqrt{b(u)}\Delta \phi(u) \right]\div\,  \boldsymbol{\eta} \ \dx \dt  \notag 
\\[3pt]
& \quad + \int_{\Omega_T} 2 b'(u) \Delta \phi(u)\nabla \phi(u) \cdot \boldsymbol{\eta} 
+ (b\Psi'')(u) \nabla u \cdot \boldsymbol{\eta} \ \dx \dt, 
\end{align*}
for all $\boldsymbol{\eta} \in C_c^\infty\left( \overline{\Omega}\times (0,T); \mathbbm{R}^3\right)$ with $\boldsymbol{\eta} \cdot \mathbf{n}=0$ on $\partial \Omega \times(0,T)$.

This proves \eqref{J} and concludes the first part of Theorem \ref{Main Result}. 

\subsection{Energy inequality}  
In light of \eqref{E3}, we know that, up to a subsequence, 
\begin{equation}
\label{J hat}
\sqrt{ b_{\ve}(u_{\ve}) }\nabla \mu_{\ve} \rightharpoonup  \widehat{\boldsymbol{J}} \quad \textrm{ in }L^2(\Omega_T).
\end{equation}
Moreover, by \eqref{pointwise conv b} and \eqref{convb}, we observe that
$$
\sqrt{ b_{\ve}(u_{\ve}) } \to \sqrt{b(u)} \quad  \textrm{ in }L^p(\Omega_T), \quad \forall \, p\in [2,\infty).
$$
Hence, we simply obtain
$$
b_\ve (u_\ve) \nabla \mu_\ve= \sqrt{b_\ve(u_\ve)}\sqrt{b_\ve(u_\ve)}\nabla \mu_\ve 
 \rightharpoonup  \sqrt{b(u)} \widehat{\boldsymbol{J}} \quad \textrm{ in }L^2(\Omega_T).
$$
On the other hand, since $\boldsymbol{J}$ is the weak limit of the approximating fluxes $b_\ve (u_\ve) \nabla \mu_\ve$, we deduce that 
$\boldsymbol{J}= \sqrt{b(u)} \widehat{\boldsymbol{J}}$ almost everywhere in $\Omega_T$.

Fix now $\psi \in C_c^\infty(0,T)$ such that $\psi \geq 0$. Multiplying \eqref{energy inequality} by $\psi$ and integrating over $(0,T)$, we find
\begin{align} 
&\int_0^T  \left[ \int_\Omega  |\nabla \phi_\ve (u_\ve(t))|^2+  \Psi(u_\ve(t)) \ \dx\right] \psi(t) \ \dt \notag
\\[3pt]
&\qquad +  \int_0^T \left[ \int_0^t \int_{\Omega} b_{\ve}(u_{\ve})|\nabla \mu_{\ve}|^2 \ \dx \mathrm{d}s \right]  \psi (t) \ \dt 
 \le \int_0^T E_{\ve}(u_0) \psi(t) \ \dt.
\label{eps energy}
\end{align}
Recalling that $E_{\ve}(u_0) \leq E(u_0)$, and by using \eqref{conv5}, \eqref{J hat}, together with $F(u_\ve) -\frac{\theta_0}{2}u^2_\ve \to F(u) -\frac{\theta_0}{2}u^2$ strongly  in $L^2(\Omega_T)$, we can pass to the (lower) limit in \eqref{eps energy} obtaining
 \begin{align} 
&\int_0^T  \left[ \int_\Omega |\nabla \phi(u(t))|^2+ \Psi(u(t))  \ \dx\right] \psi(t) \dt
\notag
\\[3pt]
&\qquad 
+  \int_0^T \left[ \int_0^t \int_{\Omega} |\widehat{\boldsymbol{J}}|^2 \ \dx \mathrm{d}s \right]  \psi (t) \dt
 \le \int_0^T E(u_0) \psi(t) \ \dt.
\label{energy}
\end{align}
In light of \eqref{nabla phiu- ae}, the latter entails 
 \begin{align} 
\int_\Omega \frac{1}{1-u^2} |\nabla u(t)|^2 + \Psi(u(t))  \ \dx
 +  \int_0^t \int_{\Omega} |\widehat{\boldsymbol{J}}|^2 \ \dx \mathrm{d}s  \leq E(u_0), \quad \text{a.e. in } (0,T),
\label{energy-2}
\end{align} 
which is exactly \eqref{energy-inequality}.
The proof of Theorem \ref{Main Result} is now completed.

\section{An equivalent weak formulation: proof of Theorem \ref{formulazione alla Cances}} 


Let $u$ be a global weak solution given by Theorem \ref{Main Result}. First of all, we
recall that
\begin{equation}
\label{nabla phiu -2}
    \nabla \phi(u)= 
    \begin{cases}
     \dfrac{1}{\sqrt{1-u^2}} \nabla u \quad &\text{ in } \{ |u|< 1 \}
     \\
     0 \quad &\text{ in } \{ |u|= 1 \}
    \end{cases}
    \quad \text{a.e. in  } \Omega_T.
\end{equation} 
 and
\begin{equation}
\label{nabla u -2}
 \nabla u = \sqrt{1-u^2} \, \nabla \phi(u)= \sqrt{b(u)} \, \nabla \phi(u), \quad \text{a.e. in } \Omega_T.
\end{equation}
Recalling the definition
$$
G(u)=\int_0^u b'(s)f(s) \ \ds \in L^\infty( \Omega_T ),
$$
we now show that
\begin{equation}
\label{nabla G point}
\nabla G(u) = 
\begin{cases} 
b'(u) f(u) \nabla u \quad &\text{ in } \{ |u|< 1 \}
     \\
     0 \quad &\text{ in } \{ |u|= 1 \}
\end{cases}
\quad \text{and} \quad 
\nabla G(u)
\in L^4(0,T; L^4(\Omega)).
\end{equation}
To this end, we first prove the integral identity
\begin{equation}
    \label{nabla G}
   \int_{\Omega_T}  \nabla G(u) \cdot \boldsymbol{\eta} \ \dx \dt
   = \int_{\Omega_T}b'(u) ( \sqrt{b}f )(u) \nabla \phi(u)\cdot \boldsymbol{\eta} \ \dx \dt 
\end{equation}
for all $\boldsymbol{\eta} \in C_c^\infty\left( \Omega \times (0,T); \mathbbm{R}^3\right)$, where the function $(\sqrt{b}f)(s)$ denotes the  continuous extension of $\sqrt{b(s)}f(s)$ in $[-1,1]$ by $(\sqrt{b}f)(\pm 1)=0$.
We proceed by approximation. Let us consider
\begin{equation*}
    u_n:= h_n(u), \quad \mbox{ where } h_n(s)=\begin{cases}
        -1+\dfrac{1}{n} \quad &\mbox{if }s< -1+\dfrac{1}{n},\\[3pt]
        s&\mbox{if } -1+\dfrac{1}{n}\le s \le 1-\dfrac{1}{n},\\[3pt]
        1-\dfrac{1}{n} \quad &\mbox{if }s> 1-\dfrac{1}{n}.
    \end{cases}
\end{equation*}
It is well known that, for $n \to + \infty$,
\begin{align}
    &u_n \to u \quad \mbox{ a.e. in }\Omega_T, \label{conv un}\\[3pt]
    &\nabla u_n = \nabla u \cdot \mathbf{1}_{\left[ -1+\frac{1}{n},1-\frac{1}{n}\right]}(u) \to \nabla u \quad \mbox{ a.e. in }\Omega_T.\label{conv nabla un}
\end{align}
In particular, we have
\begin{equation}
\label{nabla phin}
\nabla \phi(u_n)= \dfrac{1}{\sqrt{b(u_n)}}\nabla u \cdot \mathbf{1}_{\left[ -1+\frac{1}{n},1-\frac{1}{n}\right]}(u) \quad \mbox{ a.e. in }\Omega_T.
\end{equation}
We notice that
\begin{align*}
    \int_{\Omega_T}b'(u_n)\sqrt{b(u_n)}f(u_n) \nabla \phi(u_n)\cdot \boldsymbol{\eta} \ \dx \dt &= \int_{\Omega_T}b'(u_n)f(u_n) \nabla u \cdot \mathbf{1}_{\left[ -1+\frac{1}{n},1-\frac{1}{n}\right]}(u) \cdot \boldsymbol{\eta} \ \dx \dt \\[3pt]
    &= \int_{\Omega_T} \nabla G(u_n)\cdot \boldsymbol{\eta} \ \dx \dt \\[3pt]
    &= -\int_{\Omega_T}  G(u_n) \div \, \boldsymbol{\eta} \ \dx \dt. 
\end{align*}
Here we have used that $G \in C^1\left(\left[ -1+\frac{1}{n},1-\frac{1}{n}\right]\right)$.
Since $G$ is a continuous function in $[-1,1]$, owing to \eqref{conv un}, we deduce that
\begin{equation}
\label{road to nabla G 1}
\int_{\Omega_T}  G(u_n)\div \, \boldsymbol{\eta} \ \dx \dt \to \int_{\Omega_T}  G(u)\div \, \boldsymbol{\eta} \ \dx \dt.
\end{equation}
The next step consists in proving that
\begin{equation}
    \label{road to nabla G 2}
    \int_{\Omega_T}b'(u_n)\sqrt{b(u_n)}f(u_n) \nabla \phi(u_n)\cdot \boldsymbol{\eta} \ \dx \dt \to \int_{\Omega_T}b'(u) (\sqrt{b}f)(u) \nabla \phi(u)\cdot \boldsymbol{\eta} \ \dx \dt.
\end{equation}
We observe that $b'(s)$ is linear, and thus continuous. 
Hence, in view of \eqref{conv un} and of the continuity of $(\sqrt{b}f)(s)$, we have
$$b'(u_n)\sqrt{b(u_n)}f(u_n) \to b'(u) (\sqrt{b}f) (u)  \quad \mbox{ a.e. in }\Omega_T.$$
Since the sequence is uniformly bounded, it is straightforward to conclude that
\begin{equation}
    \label{conv b'sqrt b f}
    b'(u_n)\sqrt{b(u_n)}f(u_n) \to b'(u) (\sqrt{b}f) (u)  \quad \mbox{ in }L^p( \Omega_T), \quad \forall \, p < \infty.
\end{equation}
As for $\nabla \phi(u_n)$, noticing that $|\nabla \phi(u_n)| \leq |\nabla \phi(u)|$ almost everywhere by \eqref{nabla phin} and $\nabla \phi(u_n) \to \nabla \phi(u)$ almost everywhere in $\Omega_T$ by \eqref{conv nabla un}, we derive that
\begin{equation}
    \label{conv nabla phi un}
    \nabla \phi(u_n) \to \nabla \phi(u) \quad \mbox{ in }L^2( \Omega_T ).
\end{equation}
Combining \eqref{conv b'sqrt b f} and \eqref{conv nabla phi un}, we infer that \eqref{road to nabla G 2} holds. Owing to \eqref{road to nabla G 1}, \eqref{nabla G} is then proved. Finally, since $b'(u)$ and $(\sqrt{b}f)(u)$ are bounded in $\Omega_T$ and $\nabla \phi(u) \in L^4(0,T; L^4(\Omega))$, we infer that $\nabla G(u) \in L^4(0,T; L^4(\Omega))$. Lastly, in light of \eqref{nabla phiu -2}, we conclude that \eqref{nabla G point} is satisfied.

We are now in the position to prove \eqref{weak formulation Cances}.
Let $q$ be defined as \eqref{q}, 
namely 
$$
q= -2 \Delta \phi(u) + (\sqrt{b}f )(u)- \theta_0 \sqrt{b(u)}u.
$$
Fix $\eta \in C_c^\infty\left( \overline{\Omega}\times (0,T)\right)$ such that $\nabla_{\mathbf{n}} \eta=0$ on $\partial \Omega \times (0,T)$.
We choose $v= \eta$ in \eqref{weak formulation} and $\boldsymbol{\eta}=\nabla \eta$ in \eqref{J}. By combining these relations, we obtain
\begin{align*}
 \int_0^T \langle \partial_t u, \eta \rangle \  \dt 
 & =
- \int_{\Omega_T} \left[2\sqrt{b(u)}\Delta \phi(u) \right] \Delta \eta \ \dx \dt  \notag 
\\[3pt]
& \quad - \int_{\Omega_T} 2 b'(u) \Delta \phi(u)\nabla \phi(u) \cdot \nabla\eta  
+ (b\Psi'')(u) \nabla u \cdot \nabla \eta \ \dx \dt
\\[3pt]
&
=\int_{\Omega_T} \sqrt{b(u)} q \Delta \eta \ \dx \dt  
- \int_{\Omega_T} (bf )(u) \Delta \eta \ \dx \dt
+\int_{\Omega_T} \theta_0 b(u) u \Delta \eta \ \dx \dt
\notag 
\\[3pt]
& \quad + \int_{\Omega_T}  b'(u)  q \nabla \phi(u) \cdot \nabla\eta  \ \dx \dt
- \int_{\Omega_T} b'(u) (\sqrt{b}f )(u) \nabla \phi(u)\cdot \nabla \eta \ \dx \dt \notag
\\[3pt]
&\quad 
+\int_{\Omega_T} \theta_0 b'(u) \sqrt{b(u)} u \nabla \phi(u)\cdot \nabla \eta \ \dx \dt
- \int_{\Omega_T} (b\Psi'')(u) \nabla u \cdot \nabla \eta \ \dx \dt.
\end{align*}
In light of \eqref{CHI E G} and \eqref{nabla G}, we observe that
\begin{align*}
&- \int_{\Omega_T} (b f )(u) \Delta \eta \ \dx \dt
- \int_{\Omega_T} b'(u) (\sqrt{b}f )(u) \nabla \phi(u)\cdot \nabla \eta \ \dx \dt
\\[3pt]
&\quad = - \int_{\Omega_T} (b f )(u) \Delta \eta \ \dx \dt
- \int_{\Omega_T} \nabla G(u) \cdot \nabla \eta \ \dx \dt
\\[3pt]
&\quad = - \int_{\Omega_T} (b f )(u) \Delta \eta \ \dx \dt
+ \int_{\Omega_T} G(u) \Delta \eta \ \dx \dt
\\[3pt]
&\quad =-
\int_{\Omega_T}  \theta u \Delta \eta \ \dx \dt.
\end{align*}
On the other hand, by \eqref{nabla u -2} and by integrating by parts, we have
\begin{align*}
&\int_{\Omega_T} \theta_0 b(u) u \Delta \eta \ \dx \dt
+\int_{\Omega_T} \theta_0 b'(u) \sqrt{b(u)} u \nabla \phi(u)\cdot \nabla \eta \ \dx \dt
- \int_{\Omega_T} (b\Psi'')(u) \nabla u \cdot \nabla \eta \ \dx \dt
\\[3pt]
&
=
\int_{\Omega_T} \theta_0 b(u) u \Delta \eta \ \dx \dt
+\int_{\Omega_T} \theta_0 b'(u) u \nabla u \cdot \nabla \eta \ \dx \dt
- \int_{\Omega_T} 	\left(\theta -\theta_0 b(u)\right) \nabla u \cdot \nabla \eta \ \dx \dt
\\[3pt]
&=- \int_{\Omega_T} 	\theta \nabla u \cdot \nabla \eta \ \dx \dt.
\end{align*}
Therefore, the latter inequalities entail 
\begin{equation}
\label{J Cances}
 \int_0^T \langle \partial_t u, \eta \rangle \  \dt 
= 
 \int_{\Omega_T} q \left[ \sqrt{b(u)} \Delta \eta +b'(u) \nabla \phi(u) \cdot \nabla\eta \right] \ \dx \dt.
\end{equation}

The proof of Theorem \ref{formulazione alla Cances} is completed. 

\bigskip

\noindent
\textbf{Acknowledgments.} 
The authors wish to thank Andrea Poiatti for identifying an error in a previous version of the manuscript.
The authors are supported by the MUR grant Dipartimento di Eccellenza 2023-2027 of Dipartimento di Matematica, Politecnico di Milano. 
M. Conti and A. Giorgini are supported by the INdAM-GNAMPA project ``Analisi di modelli di Cahn-Hilliard per la separazione di fase'', CUP E53C25002010001.

\bigskip

\noindent
\textbf{Disclosure statement.} The authors report there are no competing interests to declare.

\medskip

\noindent
\textbf{Data availability statement.} No further data is used in this manuscript.

\end{document}